\documentclass[11pt,reqno]{amsart}
\usepackage{amssymb}
\usepackage{enumerate}
\usepackage{amsmath,amssymb,amsfonts,amsthm,
latexsym, amscd, epsfig, epic,color}
\usepackage{mathtools} 
\usepackage{extarrows} 
\usepackage{enumitem}
\usepackage{enumerate}
\usepackage{mathrsfs}
\usepackage{eucal}
\usepackage{eufrak}
\usepackage{xspace}
\usepackage{a4wide}
\usepackage{colordvi}
\usepackage{comment}
\usepackage{hyperref,eucal}
\usepackage{tikz}\usetikzlibrary{matrix,arrows}
\newcounter{mycount}
\theoremstyle{plain}
\newtheorem{theorem}[mycount]{Theorem}
\newtheorem{corollary}[mycount]{Corollary}

\newtheorem{lemma}[mycount]{Lemma}
\newtheorem{proposition}[mycount]{Proposition}
\newtheorem{conjecture}[mycount]{Conjecture}
\theoremstyle{definition}
\newtheorem{definition}{Definition}

\theoremstyle{example}
\newtheorem{example}{Example}

\theoremstyle{openproblem}
\newtheorem{openproblem}{Open Problem}

\theoremstyle{remark}
\newtheorem{remark}{Remark}
\numberwithin{equation}{section}
\numberwithin{figure}{section}

\def\des{\mathsf{des}}

\def\iasc{\mathsf{iasc}}
\newcommand{\asc}{\mathsf{asc}}
\def\max{\mathsf{max}}
\def\zero{\mathsf{zero}}
\def\rep{\mathsf{rep}}
\def\rmin{\mathsf{rmin}}
\def\lmax{\mathsf{lmax}}
\def\lmin{\mathsf{lmin}}
\def\rmax{\mathsf{rmax}}

\def\sebr{\mathsf{sebr}}
\def\Prm{\mathsf{Prm}}

\def\Rmin{\mathsf{Rmin}}

\def\ealm{\mathsf{ealm}}

\def\rpos{\mathsf{rpos}}

\def\A{\mathcal{A}}
\def\T{\CMcal{T}}
\def\M{\CMcal{M}}
\def\R{\CMcal{R}}

\def\CS{ \CMcal{S}}
\def\B{ \mathcal{B}}
\def\C{ \CMcal{C}}
\def\D{ \CMcal{D}}
\def\G{ \CMcal{G}}
\def\Gf{ \mathfrak{G}}

\def\boxit#1{\leavevmode\hbox{\vrule\vtop{\vbox{\kern.33333pt\hrule\kern1pt\hbox{\kern1pt\vbox{#1}\kern1pt}}\kern1pt\hrule}\vrule}}

\usepackage{collectbox}


\newcommand{\SEPATTERN}{
	\draw[step=1, xshift=14pt, yshift=14pt, \cfill, line cap=round] (0,0) grid (3,3);
	\draw[step=1, xshift=14pt, yshift=14pt, thick] (0,1) -- (3,1);
	\draw[step=1, xshift=14pt, yshift=14pt, thick] (1,0) -- (1,3);
	\foreach \x/\y in {1/2,2/3,3/1} \node[disc, fill=black] at (\x,\y) {};
}

\newcommand{\sepattern}{\!\raisebox{-0.3em}{
		\begin{tikzpicture}[line width=0.7pt, scale=0.15]
		\tikzstyle{disc} = [circle,thin,draw=black, minimum size=1.7pt, inner sep=0pt ]
		\SEPATTERN
		\end{tikzpicture}}
}

\newcommand{\cfill}{black!40}

\allowdisplaybreaks

\begin{document}

\title[A bi-symmetric septuple equidistribution on ascent sequences]{Proof of a bi-symmetric septuple equidistribution on ascent sequences}

\author{Emma Yu Jin and Michael J. Schlosser}
\thanks{Both authors were partially supported by the Austrian Research Fund FWF,
project P 32035.}
\address{Fakult\"{a}t f\"{u}r Mathematik, Universit\"{a}t Wien, Vienna, Austria}
\email{yu.jin@univie.ac.at}
\email{michael.schlosser@univie.ac.at}

\maketitle

\date{\today}

\begin{abstract}
It is well known since the seminal work by Bousquet-M\'elou, Claesson, Dukes and Kitaev (2010) that certain refinements of the ascent sequences with respect to several natural statistics are in bijection with corresponding refinements of $({\bf2+2})$-free posets and permutations that avoid a bivincular pattern. Different multiply-refined enumerations of ascent sequences and other bijectively equivalent structures have subsequently been extensively studied by various authors. 

In this paper, our main contributions are
\begin{itemize}
\item a bijective proof of a bi-symmetric septuple equidistribution of statistics on ascent
sequences, involving the number of ascents ($\asc$), the number of repeated entries
($\rep$), the number of zeros ($\zero$), the number of maximal entries ($\max$), the
number of right-to-left minima ($\rmin$) and two auxiliary statistics;
\item a new transformation formula for non-terminating basic hypergeometric $_4\phi_3$
series expanded as an analytic function in base $q$ around $q=1$, which is utilized to
prove two (bi)-symmetric quadruple equidistributions on ascent sequences.
\end{itemize}
A by-product of our findings includes the affirmation of a conjecture about the bi-symmetric equidistribution between the quadruples of Euler--Stirling statistics $(\asc,\rep,\zero,\max)$ and $(\rep,\asc,\max,\zero)$ on ascent sequences, that was motivated by a double Eulerian equidistribution due to Foata (1977) and recently proposed by Fu, Lin, Yan, Zhou and the first author (2018).
\end{abstract}

\section{Introduction and main results}
In the seminal paper \cite{bcdk} by Bousquet-M\'elou, Claesson, Dukes and Kitaev, ascent sequences were introduced, as they are in bijection with several different combinatorial structures such as $({\bf2+2})$-free posets, certain bivincular pattern-avoiding permutations, Stoimenow's involution and regular linearized chord diagrams \cite{sto,zag}. Several natural statistics on posets, permutations and sequences are also kept track of by a sequence of bijections established by these authors. Since then, various joint distributions of classical statistics on ascent sequences and many other bijectively equivalent structures including Fishburn matrices \cite{fi1,fi2} and $({\bf 2-1})$-avoiding inversion sequences have been intensively explored \cite{cl,dkrs,dp,je,je2,kr2,kr,lev}.

Recently, Fu, Lin, Yan, Zhou and the first author \cite{fjlyz} discovered a new decomposition of ascent sequences which contributes to a systematic study of Eulerian and Stirling statistics on ascent sequences, certain pattern-avoiding permutations and $({\bf 2-1})$-avoiding inversion sequences. In particular, their work led them to conjecture
the bi-symmetry of a quadruple
Euler--Stirling statistics on ascent sequences (see Conjecture \ref{conj1ref}) that is motivated by a double Eulerian equidistribution due to Foata \cite{fo}.
However, it appears that the use of the new decomposition from \cite{fjlyz} is not sufficient
to prove the bi-symmetry conjecture.

In the present paper, we affirm this conjecture in two different ways: one by developing a second new decomposition of ascent sequences; and the other one by identifying
the generating function of the quadruple statistics as a basic hypergeometric series
to which a new transformation formula (that is derived in this paper) is applied.
Let us start with some necessary definitions and then state the consequences of our results.

An {\em inversion sequence} $(s_1,s_2,\ldots,s_n)$ is a sequence of non-negative integers such that for all $i$, $0\le s_i<i$. We denote by $\CMcal{I}_n$ the set of inversion sequences of length $n$, which is in one-to-one correspondence with the set $\mathfrak{S}_n$ of permutations of $[n]:=\{1,2,\ldots,n\}$ via the well known Lehmer code $\sigma$ (see for instance \cite{fo,leh}). That is, for $\pi=\pi_1\pi_2\cdots \pi_n\in \mathfrak{S}_n$, the map $\sigma:\mathfrak{S}_n\rightarrow \CMcal{I}_n$ is defined as
\begin{align*}
\sigma(\pi)=(s_1,s_2,\ldots,s_n), \quad\mbox{ where } s_i:=|\{j:j<i \,\mbox{ and }\, \pi_j>\pi_i\}|.
\end{align*}
Some restrictions set up on permutations and inversion sequences could produce new sets of equal cardinality, but not necessarily through the Lehmer code. 
As a case in point, two kinds, respectively, of restricted inversion sequences and of restricted permutations, known as ascent sequences and $(\sepattern\,)$-avoiding permutations (defined as below), are equinumerous.
\begin{definition}[Ascent sequence]
	For any sequence $s\in\CMcal{I}_n$, let
	\begin{align}\label{E:asc}
	\asc(s)&:=|\{i\in[n-1]: s_i<s_{i+1}\}|
	\end{align}
	be the number of {\bf asc}ents of $s$.
	An inversion sequence $s\in \CMcal{I}_n$ is an {\em ascent sequence}
	if for all $2\leq i\leq n$, the $s_i$ satisfy
	\begin{align*}
	s_i\leq \asc(s_1,s_2,\ldots, s_{i-1})+1.
	\end{align*} 
\end{definition}
\begin{definition}[$(\sepattern\,)$-avoiding permutation]
We say that a permutation $\pi\in\mathfrak{S}_n$ {\em avoids the pattern \mbox{\sepattern}} if there is no subsequence $\pi_i\pi_{i+1}\pi_j$  of $\pi$ satisfying both $\pi_i-1=\pi_j $ and $\pi_i<\pi_{i+1}$. Otherwise we say $\pi$ contains the pattern \mbox{\sepattern}. Sometimes the pattern \mbox{\sepattern} is written as $2|3\bar{1}$.
\end{definition}
The $(\sepattern\,)$-avoiding permutations (more generally, permutations that avoid a specific bivincular pattern) were introduced and studied by Bousquet-M\'elou, Claesson, Dukes and Kitaev \cite{bcdk} as both of them are surprisingly in bijection with other classical combinatorial structures such as $({\bf2+2})$-free posets \cite{fi1,fi2} and regular linearized chord diagrams \cite{sto,zag}.

Let $\A_n$ and $\mathfrak{S}_n(\sepattern)$ be the sets respectively of ascent sequences and $(\sepattern\,)$-avoiding permutations of length $n$. Bousquet-M\'elou, Claesson, Dukes and Kitaev \cite{bcdk} proved that
\begin{equation}\label{gffb}
\left|\A_n\right|=\left|\mathfrak{S}_n(\sepattern)\right|=[t^n]\sum_{k=1}^{\infty}\prod_{i=1}^k(1-(1-t)^i),
\end{equation}
and thus, as a consequence of a result by Zagier~\cite{zag}
(who discovered that the series on the right-hand side of \eqref{gffb} is the
generating functions of the Fishburn numbers), $\left|\A_n\right|$
is equal to the $n$-th Fishburn number (see A022493 of the OEIS~\cite{oeis}).
Their explicit values are given as
\begin{equation*}
  \left(\left|\A_n\right|\right)_{n\ge 1}=
  (1, 2, 5, 15, 53, 217, 1014, 5335, 31240, 201608, \ldots),
\end{equation*}
for which no closed form is known.
The study of Fishburn numbers and their generalizations has remarkably led to many interesting results, including congruences \cite{as,gar}, asymptotic formulas \cite{BLR:14,hj,zag}, intriguing connections to transformations of hypergeometric series \cite{aj}, modular forms \cite{BLR:14,zag} and a variety of bijections \cite{cl,dkrs,dp,je,je2,kr2,kr,lev}.
In particular, various members of the Fishburn family (which are defined to be
classes of combinatorial objects enumerated by the Fishburn numbers) can be
viewed as supersets of corresponding members of the Catalan family 
(i.e., classes of combinatorial objects enumerated by the Catalan numbers,
such as non-crossing matchings, $231$-avoiding permutations and Dyck paths). 



This paper is devoted to new bijective and basic hypergeometric aspects of Fishburn structures, for which we review some classical statistics on ascent sequences and $(\sepattern\,)$-avoiding permutations.
For any sequence $s\in\CMcal{I}_n$,  $\asc(s)$ is defined in (\ref{E:asc}). Let furthermore
\begin{align*}
\rep(s)&:=n-\vert\{s_1,s_2,\ldots,s_n\}\vert,\\
\zero(s)&:=\vert\{i\in[n]: s_i=0\}\vert,\\
\max(s)&:=\vert\{i\in[n]:s_i=i-1\}\vert,\quad\text{and}\\
\rmin(s)&:=|\{s_i: s_i<s_j\text{ for all $j>i$}\}|,
\end{align*}
be the respective numbers of {\bf rep}eated entries,  {\bf zero}s, {\bf maxi}mal entries (or maximals for short) and {\bf r}ight-to-left {\bf min}ima of $s$.
For instance, when $s=(0,1,2,0,1,3,5)\in \CMcal{I}_7$, then $\asc(s)=5$, $\rep(s)=2$, $\zero(s)=2$, $\max(s)=3$ and $\rmin(s)=4$. For any permutation $\pi\in\mathfrak{S}_n$, let
\begin{align*}
\des(\pi)&:=|\{i\in[n-1]: \pi_i>\pi_{i+1}\}|,\\
\iasc(\pi)&:=\asc(\pi^{-1})=|\{i\in[n-1]:\pi_i+1 \mbox{ appears to the right of }\pi_i\}|,
\end{align*}
be the number of {\bf des}ents and {\bf i}nverse {\bf asc}ents of $\pi$, respectively. Similar to $\rmin$, the statistics $\lmin$, $\lmax$ and $\rmax$ represent the numbers of {\bf l}eft-to-right {\bf min}ima, {\bf l}eft-to-right {\bf max}ima and  {\bf r}ight-to-left {\bf max}ima, respectively.

Previous bijections developed in \cite{bcdk,dp,fjlyz} preserve natural statistics on posets, permutations, sequences and matrices. As examples, we list below five pairs of equidistributed statistics that were established in those papers.
\begin{align*}
(\asc,\zero) \mbox{ on ascent sequences}
&\xleftrightarrow{1-1} (\des,\lmax) \mbox{ on (\sepattern)-avoiding permutations},\\
&\xleftrightarrow{1-1}(\mathsf{mag},\mathsf{min}) \mbox{ on } ({\bf2+2})\mbox{-free posets},\\
&\xleftrightarrow{1-1} (\mathsf{dim}, \mathsf{rowsum}_1) \mbox{ on Fishburn matrices},\\
&\xleftrightarrow{1-1} (\rep,\max) \mbox{ on } ({\bf2-1})\mbox{-avoiding inversion sequences}.
\end{align*}
\begin{remark}
The statistics $\mathsf{mag}$, $\mathsf{min}$ are abbreviations for {\bf mag}nitude and the number of {\bf min}imal elements of a poset; the statistics $\mathsf{dim}$ and  $\mathsf{rowsum}_1$ refer to {\bf dim}ension and the {\bf sum} of entries in the {\bf first row} of a matrix. 
\end{remark}

In a recent paper \cite{fjlyz} by Fu, Lin, Yan, Zhou and the first author, a joint {\em symmetric} distribution of statistics $\asc$ and $\rep$ over ascent sequences was discovered and our motivation came from a symmetric distribution of $(\asc,\rep)$ on inversion sequences
\begin{align}\label{inv:sym}
\sum_{s\in\CMcal{I}_n}
u^{\asc(s)}x^{\rep(s)}
&=\sum_{s\in\CMcal{I}_n}
u^{\rep(s)}x^{\asc(s)},
\end{align}
which is a consequence of a double Eulerian equidistribution due to Foata~\cite{fo}:
\begin{equation}\label{dou:fo}
\sum_{s\in\CMcal{I}_n}
u^{\asc(s)}x^{\rep(s)}=\sum_{\pi\in\mathfrak{S}_n}u^{\des(\pi)}x^{\iasc(\pi)}.
\end{equation}
It turns out that not only (\ref{inv:sym}) and (\ref{dou:fo}) are true if $\CMcal{I}_n$
and $\mathfrak{S}_n$ are replaced by the corresponding subsets $\A_n$ and $\mathfrak{S}_n(\sepattern)$, but an even stronger result on a bi-symmetric equidistribution of Euler--Stirling statistics \footnote{We adopt the classification of statistics from \cite{fjlyz}: any statistic whose distribution over a member of the Fishburn family equals the distribution of $\asc$ (resp.\ $\zero$) on ascent sequences is called {\em an Eulerian} (resp. a Stirling) statistic. So according to Theorem \ref{bij:sym}, $\asc,\rep$ are Eulerian statistics and $\zero,\max,\rmin$ are Stirling statistics.}  over ascent sequences appears to hold,
supported by experimental data.
\begin{conjecture}\label{conj1ref}\cite{fjlyz}
For each $n\geq1$, the following bi-symmetric quadruple equidistribution holds:
	\begin{align*}
	\sum_{s\in\A_n}
	u^{\asc(s)}x^{\rep(s)}z^{\zero(s)}y^{\max(s)}
	&=\sum_{s\in\A_n}
	u^{\rep(s)}x^{\asc(s)}z^{\max(s)}y^{\zero(s)}. 
	\end{align*}
\end{conjecture}
\begin{remark}
Conjecture \ref{conj1ref} is equivalent to a bi-symmetric equidistribution between the quadruples $(\des,\iasc,\lmax,\lmin)$ and $(\iasc,\des,\lmin,\lmax)$ on $(\sepattern)$-avoiding permutations, according to Theorem 12 of \cite{fjlyz}.
\end{remark}
Two results in approaching this conjecture were presented in \cite{fjlyz}: one is a generating function formula of ascent sequences with respect to the statistics $\asc,\rep,\zero,\max$ (see Theorem \ref{T:gen}); and the other one is a quadruple equidistribution between $(\asc,\rep,\zero,\max)$ and $(\rep,\asc,\rmin,\zero)$ on ascent sequences (see Theorem \ref{bij:sym}).

 Let $\G(t;x,y,u,z)$ denote the generating function of ascent sequences counted by the length (variable $t$), $\asc$ (variable $u$), $\rep$ (variable $x$), $\max$ (variable $y$) and $\zero$ (variable $z$). That is,
\begin{align}\label{E:g1}
\G(t;x,y,u,z):=\sum_{n=1}^{\infty}t^n\sum_{s\in\A_n}x^{\rep(s)}y^{\max(s)}u^{\asc(s)}z^{\zero(s)}.
\end{align}
\begin{theorem}\label{T:gen}\cite{fjlyz}
	The generating function $\G(t;x,y,u,z)$ of ascent sequences is
	\begin{align}\label{E:genG}
	\G(t;x,y,u,z)=\sum_{m=0}^{\infty}&\frac{zyr x^m(1-yr)(1-r)^m(x+u-xu)}
	{[x(1-u)+u(1-yr)(1-r)^m][x+u(1-x)(1-yr)(1-r)^m]}\notag\\
	&\times \prod_{i=0}^{m-1}\frac{1+(zr-1)(1-yr)(1-r)^i}{x+u(1-x)(1-yr)(1-r)^i},
	\end{align}
	where $r=t\,(x+u-xu)$.
\end{theorem}
\begin{theorem}\label{bij:sym}\cite{fjlyz}
	There is a bijection $\Upsilon:\A_n\rightarrow\A_n$ which transforms the quadruple
	$$
	(\asc,\rep,\zero,\max)\; \text{ to }\; (\rep,\asc,\rmin,\zero).
	$$
\end{theorem}
Conjecture \ref{conj1ref} can be settled, with the the help of Theorems \ref{T:gen} and \ref{bij:sym}, by showing either (I) or (II), described as follows.

\begin{enumerate}[label=(\Roman*)]
	\item $\G(t;x,y,u,z)=\G(t;u,z,x,y)$;
	\item the quadruple $(\asc,\rep,\zero,\max)$ has the same distribution as $(\asc,\rep,\zero,\rmin)$ over ascent sequences. 
\end{enumerate}
In this paper, we settle Conjecture \ref{conj1ref} in both ways, (I) and (II).

Our first main result (Theorem \ref{T:5tuple}) is a bijective proof of
a bi-symmetric septuple equidistribution on ascent sequences,
which significantly generalizes (II) and consequently affirms Conjecture \ref{conj1ref}.
\begin{theorem}\label{T:5tuple}
	There is a bijection $\Phi: \A_n\rightarrow \A_n$ such that for all $s\in \A_n$, 
	\begin{align}\label{E:thm7sta}
	(\asc,\rep,\zero,\max,\ealm,\rmin,\rpos)s=(\asc,\rep,\zero,\rmin,\rpos,\max,\ealm)\Phi(s).
	\end{align}
\end{theorem}
We postpone the definitions of the two auxiliary statistics $\ealm,\rpos$ to Sections \ref{S:thm3} and \ref{S:bij}. 
The main idea to prove Theorem \ref{T:5tuple} relies on two {\em parallel} decompositions of ascent sequences that are in close relation to the two respective
auxiliary statistics $\ealm$ and $\rpos$. The former decomposition was discovered in \cite{fjlyz}. However, using this decomposition alone appears to be not enough to prove Conjecture \ref{conj1ref}, which motivates us to develop the latter new decomposition in this paper, providing a crucial piece of the puzzle solved here.



Our second main result (Theorem \ref{thm:4phi3tf}) is a new transformation formula
of non-terminating basic hypergeometric ${}_4\phi_3$ series, valid as an identity
expanded in base $q=1-r$ around $q=1$, or, equivalently, $r=0$,
with relevant definitions being given in Section \ref{S:hyp}.
\begin{theorem}\label{thm:4phi3tf}
	Let $a,b,c,d,e,r$ be complex variables, $j$ be a non-negative integer.
	Then, assuming that none of the denominator factors in \eqref{tf43}
	have vanishing constant term in $r$, we have the following transformation
	of convergent power series in $a$ and $r$:
	\begin{align}\label{tf43}
	&{}_4\phi_3\!\left[\begin{matrix}(1-r)^j,1-a,b,c\\
	d,e,(1-r)^{j+1}(1-a)bc/de\end{matrix};1-r,1-r\right]\notag\\
	&=\frac{((1-r)/e,(1-r)(1-a)bc/de;1-r)_j}{((1-r)(1-a)/e,(1-r)bc/de;1-r)_j}
	\notag\\
	&\quad\;\times{}_4\phi_3\!\left[\begin{matrix}(1-r)^j,1-a,d/b,d/c\\
	d,de/bc,(1-r)^{j+1}(1-a)/e\end{matrix};1-r,1-r\right].
	\end{align}
\end{theorem}
We utilize special cases of Theorem \ref{thm:4phi3tf} to give analytic proofs
of two different quadruple (bi)-symmetric equidistributions of Euler--Stirling statistics
on ascent sequences, collected in Theorem~\ref{TC:int}.
The first application of Theorem \ref{thm:4phi3tf} is a proof of (I)
by making use of the explicit form of the generating function in Theorem~\ref{T:gen},
and thus constitutes a non-combinatorial proof of the bi-symmetric equidistribution
in Conjecture \ref{conj1ref}, while the second application establishes a symmetric
equidistribution by employing a new explicit generating function obtained by
a refined recursive construction of ascent sequences from \cite{fjlyz}.
\begin{theorem}\label{TC:int}
For the generating function defined in (\ref{E:g1}), we have the bi-symmetry
\begin{equation}\label{E:sym1}
\G(t;x,y,u,z)=\G(t;u,z,x,y).
\end{equation}
Furthermore, define 
	\begin{equation}\label{E:g2}
	\Gf(t;x,y,u,v):=\sum_{n=1}^{\infty}t^n\sum_{s\in\A_n}x^{\rep(s)}y^{\max(s)}
	u^{\asc(s)}v^{\rmin(s)},
	\end{equation}
	then we have, with $r=t(x+u-xu)$,
	\begin{align}\label{E:g5}
	\Gf(t;x,y,u,v)=\frac{vyt}{1-vytu}
	+\sum_{m=0}^{\infty}\frac{rv(1-yr)(1-r)^m}
	{(x-xu+u(1-yr)(1-r)^m)(1-tuvy)}&\notag\\
	\times\prod_{i=0}^{m}\frac{x(1-(1-yr)(1-r)^i)(x-xu+u(1-yr)(1-r)^i)}
	{(x-u(x-1)(1-yr)(1-r)^i)(x-xu+u(1-rv)(1-yr)(1-r)^i)}&,
	\end{align}
and the symmetry
	\begin{equation}\label{E:sym2}
	\Gf(t;x,y,u,v)=\Gf(t;x,v,u,y),
	\end{equation}
	In other words, for any ascent sequence $s\in\A_n$, 
	\begin{align*}
	(\asc,\rep,\zero,\max)s&=(\rep,\asc,\max,\zero)\Upsilon^{-1}(\Phi(s)),\\
	(\asc,\rep,\max,\rmin)s&=(\asc,\rep,\rmin,\max)\Phi(s),\\
	(\asc,\rep,\zero,\rmin)s&=(\rep,\asc,\rmin,\zero)\Upsilon(\Phi(s)),
	\end{align*}
	where $\Upsilon$ and $\Phi$ are the bijections respectively in Theorems \ref{bij:sym} and \ref{T:5tuple}. 
\end{theorem}

\begin{remark}
	We are not the first ones to study equivalent forms for generating functions of objects of the Fishburn family using tools from basic hypergeometric series. Initiating with work of Zagier~\cite{zag} who established the basic hypergeometric series in \eqref{gffb} as a concrete form of the generating function $\G(t;1,1,1,1)$ for the Fishburn numbers, Andrews and Jel\'{i}nek \cite{aj} subsequently proved three equivalent forms of $\G(t;1,1,1,z)$ by applying the Rogers--Fine identity. However, to the best of our knowledge, no algebraic or analytic arguments to determine equivalent forms of the generating functions $\G(t;x,y,u,z)$
or $\Gf(t;x,y,u,v)$ were known, not even, say, for the special case $\G(t;1,1,u,z)$.
Our analytic proofs of $\G(t;x,y,u,z)=\G(t;u,z,x,y)$ and $\Gf(t;x,y,u,v)=\Gf(t;x,v,u,y)$
strengthen the already known existing ties between (refined) generating functions of objects of the Fishburn family with basic hypergeometric series that are expanded in base $q=1-r$ around $r=0$. At the same time it demonstrates the benefit of having equivalent forms of generating functions, and the power of basic hypergeometric machinery.
\end{remark}
All aforementioned (bi)-symmetric distributions on ascent sequences have counterparts over other members of the Fishburn family.
\begin{corollary}
There are three bijections between $\mathfrak{S}_n(\sepattern)$ and itself such that the following three (bi)-symmetric equidistributions hold, respectively:
		\begin{align*}
	(\des,\iasc,\lmax,\lmin,\rmax)\pi&=(\des,\iasc,\lmax,\rmax,\lmin)(\Psi^{-1}\circ\Phi\circ\Psi)(\pi),\\
	(\des,\iasc,\lmax,\lmin)\pi&=(\iasc,\des,\lmin,\lmax)(\Psi^{-1}\circ\Upsilon^{-1}\circ\Phi\circ\Psi)(\pi),\\
	(\des,\iasc,\lmax,\rmax)\pi&=(\iasc,\des,\lmin,\lmax)(\Psi^{-1}\circ\Upsilon\circ\Phi\circ\Psi)(\pi),
	\end{align*}
	 where $\Upsilon,\Phi$ are the bijections respectively in Theorems \ref{bij:sym} and \ref{T:5tuple}, and $\Psi:\mathfrak{S}_n(\sepattern)\rightarrow \A_n$ is the bijection from Theorem 12 of \cite{fjlyz}.
\end{corollary}
Let us recall the definition of Fishburn matrices and associated three Stirling statistics.

Any cell $(i,j)$ of a matrix $M$ is called a {\em weakly north-east cell}\/ if $M_{i,j}\ne 0$ and $M_{s,t}=0$ for all $s\le i$ and $t\ge j$. A matrix is a {\em Fishburn} matrix if all of its entries are non-negative integers such that neither row nor column contains only zero entries. Let $\CMcal{F}_n$ be the set of Fishburn matrices whose sum of entries equals $n$, then for any $M\in\CMcal{F}_n$, let 
\begin{align*}
\mathsf{rowsum}_1(M)&:=\mbox{ the sum of entries in the first row of } M,\\
\mathsf{ne}(M)&:=\mbox{ the number of weakly north-east cells of } M, \\
\mathsf{tr}(M)&:=\mbox{ the number of non-zero cells on the main diagonal of } M.
\end{align*}
\begin{corollary}\label{C:fish}
	There is a bijection between $\CMcal{F}_n$ and itself such that the following symmetric distribution holds:
	\begin{align*}
	(\mathsf{rowsum}_1,\mathsf{ne},\mathsf{tr})M=(\mathsf{rowsum}_1,\mathsf{tr},\mathsf{ne})(\phi\circ \Phi\circ\phi^{-1})M,
	\end{align*}	
	 where $\Phi$ is the bijection in Theorem \ref{T:5tuple} and $\phi:\A_n\rightarrow \CMcal{F}_n$ is the bijection from Theorem 3.6 of \cite{cyz}.
\end{corollary}
\begin{remark}	
The three Stirling statistics $\mathsf{rowsum}_1,\mathsf{ne}, \mathsf{tr}$ are pairwise symmetric on $\CMcal{F}_n$. The fact that the pair $(\mathsf{ne}, \mathsf{tr})$ is symmetric on $\CMcal{F}_n$ is a direct consequence of Corollary \ref{C:fish} and it is known from  \cite{cyz,je2} that the other two pairs $(\mathsf{rowsum}_1,\mathsf{ne})$ and $(\mathsf{rowsum}_1,\mathsf{tr})$ are also symmetric.
\end{remark}

The paper is organized as follows. In Section~\ref{S:thm3} and \ref{S:bij}, we present a new decomposition of ascent sequences and a sequence of relevant bijections. In Section~\ref{Sec:7tuple} we prove Theorem~\ref{T:5tuple}. A refined generating function of ascent sequences is derived in Section~\ref{S:refgf} and its amenability to transformations of basic hypergeometric series is demonstrated in Section~\ref{S:hyp}.
We end the paper in Section \ref{S:fre} with some final remarks;
in particular we pose an open problem there and state
a conjecture on a symmetric quintuple equidistribution
over inversion sequences.



\section{A new decomposition of ascent sequences}\label{S:thm3}
The purpose of this section is to develop a new decomposition of ascent sequences. This decomposition is centered around a new auxiliary statistic $\rpos$, which is defined as below.

Let $\Rmin$ be the corresponding set-valued statistic of $\rmin$, that is, for any $s\in\CMcal{I}_n$,
\begin{align*}
\mathsf{Rmin}(s)=\{s_i:s_i<s_j \mbox{ for all } j>i\}.
\end{align*}
For the sake of convenience, we index all right-to-left minima from left to right starting from 0 (rather than from 1). That is, all right-to-left minima of $s$ are indexed by $0,1,\ldots,\rmin(s)-1$ from left to right.

\begin{definition}[statistics $\rpos$]\label{D:rpos}
	For any ascent sequence $s$ with $\rmin(s)\ne |s|$, define $\rpos(s)=m$ if $m$ is the maximal index such that the $m$-th right-to-left minimum appears at least twice after the $(m-1)$-th right-to-left minimum. 
	If no such $m$ exists or $\rmin(s)=|s|$, set $\rpos(s)=0$.
	
	For example, $\rpos({\bf 0},{\bf 0},1,2,3,4)=0$ and $\rpos(0,0,1,2,0,1,2,1,{\bf 3},{\bf 3},4)=2$.
\end{definition} 


Let $\A^*$ denote the set of all ascent sequences except $s=(0,1,2,\ldots,|s|-1)$, we will divide it into a couple of disjoint subsets. Some statistics are needed in order to describe this new decomposition.

\begin{definition}[statistics $\sebr$]\label{D:sebr}
	Let $\Rmin(s)_j$ be the $j$-th smallest element of $\Rmin(s)$ where $0\le j< \rmin(s)$. Given any ascent sequence $s\in\A^*$, define $\sebr(s)$ to be the {\bf s}mallest {\bf e}ntry {\bf b}etween the two {\bf r}ightmost entries $\Rmin(s)_{\rpos(s)}$, and assume $\sebr(s)=0$ if the two rightmost entries $\Rmin(s)_{\rpos(s)}$ are next to each other.
	
	For example, let $s=(0,0,1,2,0,1,2,1,{\bf 3},4,5,{\bf 3},4)$, then $\rpos(s)=2$, $\Rmin(s)_{2}=3$ and $\sebr(s)=4$ where the two rightmost entries $3$ are in bold.
\end{definition} 
We set
\begin{align*}
\T_{1}&:=\{s\in\A^*: |s|=\rmin(s)+1\},\\
\T_{2}&:=\{s\in\A^*-\T_1: \sebr(s)=0\},
\end{align*}
so that the complement $\T_1\dot\cup\T_2$ in $\A^*$ contains all ascent sequences $s\in \A^*$ with $\sebr(s)\ne 0$. We next divide the remaining set
$\A^*-\T_1\dot\cup\T_2$ into the following two disjoint subsets $\A^1$ and $\A^2$ by comparing $\sebr(s)$ and $\Rmin(s)_{\rpos(s)+1}$. For the case $\rpos(s)=\rmin(s)-1$, we always assume that $\sebr(s)<\Rmin(s)_{\rpos(s)+1}$. 
\begin{align*}
\A^1&:=\{s\in\A^*:  \sebr(s)\ge \Rmin(s)_{\rpos(s)+1}\},\\
\A^2&:=\{s\in\A^*: 0\ne \sebr(s)<\Rmin(s)_{\rpos(s)+1}\}.
\end{align*}
Now we refine the sets $\A^1$ and $\A^2$ through the concept of {\em maximal ascents}:

\begin{definition}($\CMcal{M}$asc)\label{D:masc}
	An entry $s_i$ of an ascent sequence $s$ is an {\bf $\CMcal{M}$}aximal {\bf asc}ent ({\em  $\CMcal{M}$asc}) if $$s_i=\asc(s_1,s_2,\ldots,s_{i-1})+1.$$ In particular, the last entry $s_{|s|}$ of $s$ is an $\M asc$ if $s_{|s|-1}<s_{|s|}=\asc(s)$. 
\end{definition}
\begin{remark}
	By Definition \ref{D:masc}, all maximal entries are $\M asc$'s. For instance, given $s=({\bf 0},{\bf 1},{\bf 2},0,{\bf 3},2)$, the entries in bold are all $\CMcal{M}$asc's of $s$.
\end{remark}
Let $\Prm(s)_j$ be the $j$-th smallest element of $\Prm(s)$, where 
\begin{align*}
\Prm(s):=\{i: s_i \mbox{ is a right-to-left minimum of } s\}
\end{align*}
is the set of {\bf p}ositions of {\bf r}ight-to-left {\bf m}inima of $s$,  we then further partition $\A^1$ and $\A^2$ into the following disjoint subsets (see  Figure \ref{F:b1}):
\begin{align*}
\T_{3}&:=\{s\in\A^1: \sebr(s)=\Rmin(s)_{\rpos(s)+1}, \Prm(s)_{\rpos(s)+1}=\Prm(s)_{\rpos(s)}+1, \\
&\quad\quad\qquad\qquad\qquad \mbox{ and no $\CMcal{M}$asc appears after the position } \Prm(s)_{\rpos(s)}\}, \\
\T_{4}&:=\{s\in\A^2: 0\ne \sebr(s)<\Rmin(s)_{\rpos(s)+1} \mbox{ and } s_{|s|} \mbox{ is not an }\M asc \},\\
\T_{5,1}&:=\{s\in\A^1: \sebr(s)>\Rmin(s)_{\rpos(s)+1} \mbox{ and } \Prm(s)_{\rpos(s)+1}=\Prm(s)_{\rpos(s)}+1\},\\
\T_{5,2}&:=\{s\in\A^1: \sebr(s)\ge \Rmin(s)_{\rpos(s)+1} \mbox{ and } \Prm(s)_{\rpos(s)+1}\ne \Prm(s)_{\rpos(s)}+1\},\\
\T_{5,3}&:=\{s\in\A^2: 0\ne \sebr(s)<\Rmin(s)_{\rpos(s)+1}\}-\T_4,\\
\T_{5,4}&:=\{s\in\A^1: \sebr(s)=\Rmin(s)_{\rpos(s)+1}, \Prm(s)_{\rpos(s)+1}=\Prm(s)_{\rpos(s)}+1\}-\T_3.
\end{align*}
Since $\T_3\dot\cup \T_{5,1}\dot\cup \T_{5,2}\dot\cup \T_{5,4}=\A^1$, $\T_4\dot\cup \T_{5,3}=\A^2$ and $\T_1\dot \cup\T_2=\A^*-(\A^1\dot\cup \A^2)$, it is clear that $\A^*$ is the disjoint union of subsets $\T_i, \T_{5,i}$ for $1\le i\le 4$.
\begin{figure}[ht]
	\centering
	\includegraphics[scale=0.7]{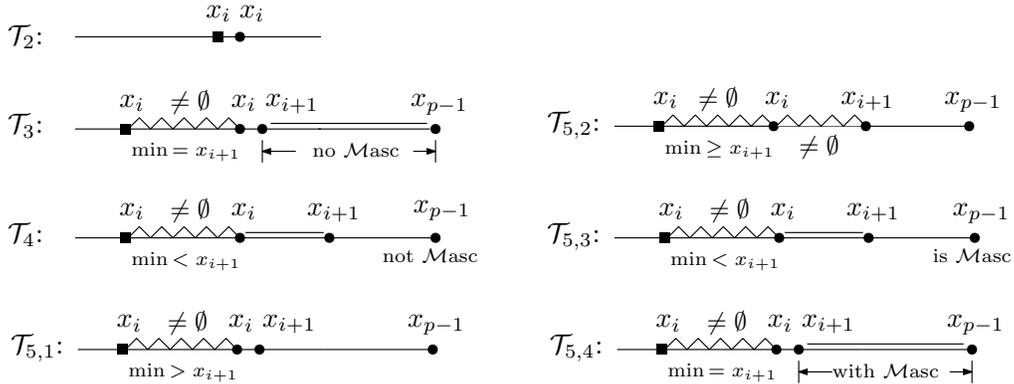}
	\caption{A partition of the set of ascent sequences $s\in\A^*$ with $\rpos(s)=i$ and $\rmin(s)=p$, where $x_i=\Rmin(s)_{i}$ denotes the $i$-th right-to-left minimum of $s$; black dots and squares represent the rightmost and the second rightmost entry respectively. \label{F:b1}}
\end{figure}
\section{A sequence of bijections on ascent sequences}\label{S:bij}
In this section, we present a sequence of bijections that map each $\T_i,\T_{5,i}$ where $i\ne 1$ to a subset of ascent sequences $s$ either with smaller $|s|-\rmin(s)$ or with smaller $\rmin(s)-\rpos(s)$. Then, by combining the parallel bijections from \cite{fjlyz}, we are able to recursively construct the desired bijection $\Phi$ for Theorem \ref{T:5tuple} in Section \ref{Sec:7tuple}.

Let us recall the statistic $\ealm$ introduced in \cite{fjlyz}.
\begin{definition}[statistic $\ealm$]\label{D:ealm}
	
	Let $s$ be an ascent sequence with $\mathsf{max}(s)\ne |s|$, then $\ealm(s)=s_{\max(s)+1}$, i.e., the entry right after the last maximal.	If $s=(0,1,\ldots,|s|-1)$ is an ascent sequence with $\max(s)=|s|$, then we set $\ealm(s)=0$.
	
	For example, $\ealm(0,1,{\bf 0},1,3,0,2)=0$.
\end{definition} 

Throughout the paper, define $\chi(a)=1$ if the statement $a$ is true; and $\chi(a)=0$ otherwise.

\begin{lemma}\label{L:rmincase2}
	There is a bijection
	$$f_2:\T_2\cap \A_n\rightarrow \{(i,s):s\in\A^*\cap \A_{n-1}, \rpos(s)\le i< \rmin(s)\}$$ 
	that sends $s$ to a pair $f_2(s)=(\rpos(s),s^*)$ satisfying
	\begin{equation*}
	(\asc,\max,\ealm, \rmin)s=(\asc,\max,\ealm,\rmin)s^*,
	\end{equation*}
	\begin{equation*}
	\zero(s)=\zero(s^*)+\chi(\rpos(s)=0),\quad\;
	\text{and}\quad\;
	\rep(s)=\rep(s^*)+1.
	\end{equation*}
\end{lemma}
\begin{proof}
	For any ascent sequence $s \in\T_2$ with $\rpos(s)=i<\rmin(s)$, the two rightmost $\Rmin(s)_{i}$ are next to each other; see $\T_2$ from Figure \ref{F:b1}. Removing one of them leads to an ascent sequence $s^*\in \A^*$ with $\rpos(s^*)\le i$. We set $f_2(s)=(\rpos(s),s^*)$ and it is easily seen that $f_2$ is a bijection satisfying
	$|s^*|=|s|-1$,
	$\asc(s^*)=\asc(s)$, $\rep(s^*)=\rep(s)-1$, $\zero(s^*)=\zero(s)-\chi(\rpos(s)=0)$, $\max(s^*)=\max(s)$, $\ealm(s^*)=\ealm(s)$ and $\rmin(s^*)=\rmin(s)$. 
\end{proof}
\begin{example}
	For $s=(0,0,1,2,0,1,2,1,{\bf3},{\bf3},4)$, according to Lemma \ref{L:rmincase2}, $f_2(s)=(2,s^*)$ where $s^*=(0,0,1,2,0,{\bf 1},2,{\bf 1},3,4)$ is an ascent sequence with $\rpos(s^*)=1$.
\end{example}

Let $\CMcal{P}_1$ be the set of ascent sequences $s\in \A^*$ whose last entry is an $\M asc$, that is, $s_{|s|-1}<s_{|s|}=\asc(s)$. Denote by $\CMcal{P}_1^c$ the complement of $\CMcal{P}_1$ in $\A^*$.
\begin{lemma}\label{L:p1}
	There is a bijection $$\phi_1:\A_n \cap \CMcal{P}_1\rightarrow \A_{n-1}\cap \A^*$$ 
	that transforms the septuple 
	$$(\asc,\rep,\zero,\max,\ealm,\rmin,\rpos) \mbox{ to } (\asc+1,\rep,\zero,\max,\ealm,\rmin+1,\rpos).$$
\end{lemma}
\begin{proof}
	For any ascent sequence $s\in \CMcal{P}_1$, remove the last entry and define the resulting sequence as $\phi_1(s)$. It is easy to examine the corresponding statistics.
\end{proof}

\begin{lemma}\label{L:rmincase5}
	There is a bijection 
	\begin{align*}
	f_3:\T_3 \cap \A_n \rightarrow \{s\in \A_{n} \cap \CMcal{P}_1:\rpos(s)\ne 0\}
	\end{align*}
	that transforms the quintuple 
	\begin{align*}
	(\asc,\rep,\max,\rmin,\rpos) \mbox{ to } (\asc,\rep+1,\max,\rmin-1,\rpos-1),
	\end{align*}
      and satisfies
	\begin{align*}
	\zero(s)&=\zero(f_3(s))+\chi(\rpos(s)=0),\\
	\ealm(s)&=\ealm(f_3(s))-\chi(\Prm(s)_{\rpos(s)}=\max(s)+1).
	\end{align*}
\end{lemma}
\begin{proof}
	For any ascent sequence $s\in\T_3$ with $\rpos(s)=i$, remove the rightmost $\Rmin(s)_{i}$ and add the integer $\asc(s)$ at the end. Let $f_3(s)$ be the resulting sequence and the map $f_3$ is clearly a bijection (see Figure \ref{F:b2}). Only when the entry $\ealm(s)$ on the $(\max(s)+1)$-th position of $s$ is also the $\rpos(s)$-th right-to-left minimum, we have $\ealm(f_3(s))=\ealm(s)+1$. It is not hard to verify other statistics.
\end{proof}
\begin{lemma}\label{L:rmincase3}
	There is a bijection 
	\begin{align*}
	f_4:\T_4\cap \A_n\rightarrow &\{s\in \A_n\cap \CMcal{P}_1^c: \rpos(s)\ne 0\}
	\end{align*}
	that transforms the quintuple 
	\begin{align*}
	(\asc,\rep,\max,\rmin,\rpos) \mbox{ to } (\asc,\rep,\max,\rmin-1,\rpos-1),
	\end{align*}
	and satisfies
	\begin{align*}
	\zero(s)&=\zero(f_4(s))+\chi(\rpos(s)=0),\\
	\ealm(s)&=\ealm(f_4(s))-\chi(\Prm(s)_{\rpos(s)}=\max(s)+1).
	\end{align*}
\end{lemma}
\begin{proof}
	For any ascent sequence $s \in\T_4\cap \A_n$ with $\rpos(s)=i$, replacing the rightmost $\Rmin(s)_{i}$ by the integer $\sebr(s)$ yields an ascent sequence $f_4(s)\in\CMcal{P}_1^c$. The map $f_4$ is invertible and therefore bijective (see Figure \ref{F:b2}). Similar to Lemma \ref{L:rmincase5}, it is straightforward to check the corresponding statistics.
\end{proof}
\begin{example}
	For $s=(0,0,1,2,0,{\bf 1},2,{\bf 1},2,4,3,5)$, then $f_3(s)=(0,0,1,2,0,1,{\bf 2},{\bf 2},4,3,5,7)$, which, according to Lemma \ref{L:rmincase5}, is an ascent sequence with $\rpos(f_3(s))=2$ and the last entry $7$ is an $\M asc$. For $\tilde{s}=(0,0,1,2,0,{\bf 1},2,{\bf 1},4,3,5)$, then $f_4(\tilde{s})=(0,0,1,2,0, 1,{\bf 2},{\bf 2},4,3,5)$. By Lemma \ref{L:rmincase3}, it is an ascent sequence with $\rpos(f_4(\tilde{s}))=2$ and the last entry $5$ is not an $\M asc$.
\end{example}
\begin{figure}[ht]
	\centering
	\includegraphics[scale=0.65]{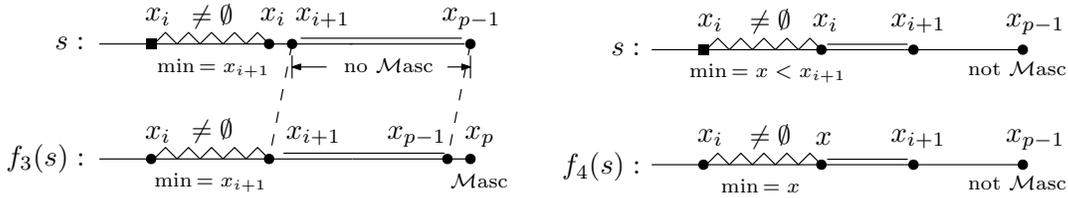}
	\caption{Two bijections $f_3$ and $f_4$ on ascent sequences $s\in \T_3 \dot\cup\T_4$ with $\rpos(s)=i$ and $\rmin(s)=p$. Here $x_i$ always denotes the  entry of the $i$-th right-to-left minimum, so $x_{p-1}$ is the last entry of $s$.\label{F:b2}}
\end{figure}
\begin{lemma}\label{L:f53}
	There is a bijection $f_{5,1}$ between the set $\T_{5,1}\cap \A_n$ and the set of ascent sequences $s\in \A_n$ such that $\rpos(s)\ne 0$ and
	\begin{itemize}
		\item the rightmost $\Rmin(s)_{\rpos(s)-1}$ is next to the second rightmost $\Rmin(s)_{\rpos(s)}$;
		\item the two rightmost $\Rmin(s)_{\rpos(s)}$ are not next to each other and there exists no $\M asc$ in between.
	\end{itemize}
	In addition, the bijection $f_{5,1}$ sends the septuple 
	\begin{equation*}
	(\asc,\rep,\max,\ealm,\rmin,\rpos) \mbox{ to } (\asc,\rep,\max,\ealm,\rmin,\rpos-1),
	\end{equation*}
	and satisfies
	\begin{equation*}
	\zero(s)=\zero(f_{5,1}(s))+\chi(\rpos(s)=0).
	\end{equation*}
\end{lemma}
\begin{proof}
	For any ascent sequence $s\in \T_{5,1}$ with $\rpos(s)=i$, insert $\Rmin(s)_{i+1}$ right after the second rightmost $\Rmin(s)_i$ and remove the rightmost $\Rmin(s)_i$; see Figure \ref{F:c4}. Define the resulting sequence as $f_{5,1}(s)$. It is easily seen that $f_{5,1}$ is a bijection and  it fulfills all properties listed in this lemma.
\end{proof}
\begin{lemma}\label{L:f54}
	There is a bijection between $f_{5,2}$ between the set $\T_{5,2}\cap \A_n$ and the set of ascent sequences $s\in \A_n$ such that $\rpos(s)\ne 0$ and
	\begin{itemize}
		\item the rightmost $\Rmin(s)_{\rpos(s)-1}$ is not next to the second rightmost $\Rmin(s)_{\rpos(s)}$;
		\item the two rightmost $\Rmin(s)_{\rpos(s)}$ are not next to each other.
	\end{itemize}
	In addition, the bijection $f_{5,2}$ sends the quintuple 
	\begin{equation*}
	(\asc,\rep,\max,\rmin,\rpos) \mbox{ to } (\asc,\rep,\max,\rmin,\rpos-1),
	\end{equation*}
	and satisfies
	\begin{align*}
	\zero(s)&=\zero(f_{5,2}(s))+\chi(\rpos(s)=0),\\
	\ealm(s)&=\ealm(f_{5,2}(s))-\chi(\Prm(s)_{\rpos(s)}=\max(s)+1).
	\end{align*}
\end{lemma}
\begin{proof}
	For any ascent sequence $s\in \T_{5,2}$ with $\rpos(s)=i$, replace the rightmost $\Rmin(s)_{\rpos(s)}$ by $\Rmin(s)_{\rpos(s)+1}$; see Figure \ref{F:c4}. Define the resulting sequence to be $f_{5,2}(s)$. It is straightforward to verify the change of statistics. 	
\end{proof}
\begin{figure}[ht]
	\centering
	\includegraphics[scale=0.7]{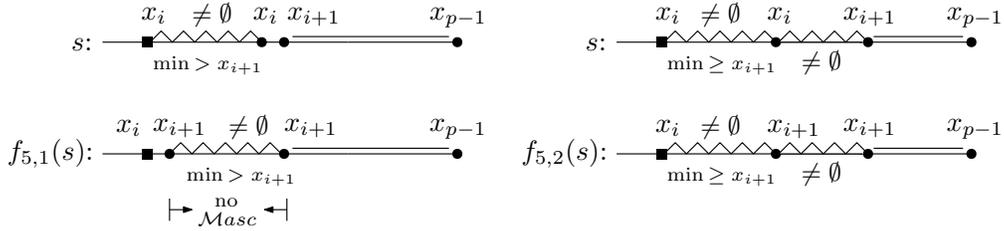}
	\caption{The bijections $f_{5,1}$ and $f_{5,2}$ in Lemma \ref{L:f53} and \ref{L:f54}. Here $x_i=\Rmin(s)_{i}$ and $i=\rpos(s)$. \label{F:c4}}
\end{figure}

The next bijection on the subset $\T_{5,3}\dot\cup\T_{5,4}$ (see Proposition \ref{P:rmin4}) is more sophisticated and a new statistic $\min \CMcal{M}asc$ is defined in order to describe the image set.
\begin{definition} (statistic $\min \CMcal{M}asc$)
	For any ascent sequence $s$, define $\min\M asc(s)$ to be the {\bf min}imal {\bf $\M asc$} (see Definition \ref{D:masc}) between the two rightmost entries $\Rmin(s)_{\rpos(s)}$. If no such $\M asc$ exists, then we assume $\min \M asc(s)=0$.
	
	For example, given $s=(0,1,2,1,3,{\bf 4},4,3,5)$, we have $\rpos(s)=2$ and $\min \M asc(s)=4$ because $4$ is the minimal $\M asc$ between the two rightmost entries $\Rmin(s)_2=3$.
\end{definition}	

\begin{proposition}\label{P:rmin4}		
	There is a bijection $f_{5}^*$ between the set $(\T_{5,3}\dot\cup\T_{5,4})\cap \A_n$  and the set $\B$ of ascent sequences $s\in \A^*\cap\A_n$ with the following properties:
	\begin{itemize}
		\item $\rpos(s)\ne 0$ and $\min \M asc(s)\ne 0$;
		\item the rightmost $\Rmin(s)_{\rpos(s)-1}$ is next to the second rightmost $\Rmin(s)_{\rpos(s)}$;
	\end{itemize}
	Furthermore, $f_{5}^*$ transforms the quadruple
	\begin{equation*}
	(\asc,\rep,\rmin,\rpos) \textrm{ to }(\asc,\rep,\rmin,\rpos-1),
	\end{equation*}
	and satisfies
	\begin{equation*}
	\zero(s)=\zero(f_{5}^*(s))+\chi(\rpos(s)=0).
	\end{equation*}
	If the second rightmost $\Rmin(s)_{\rpos(s)}$ is a maximal of $s$, then $f_5^*$ transforms the pair $$(\max,\ealm)\mbox{ to }(\max-1,\ealm-1);$$ otherwise it transforms the pair $$(\max,\ealm)\mbox{ to }(\max,\ealm).$$
\end{proposition}
We divide Proposition \ref{P:rmin4} into two Lemmas (Lemma \ref{L:rmin4} and \ref{L:rmin5}) and prove them in subsection \ref{ss:r12} and \ref{ss:r34} separately as the proofs employ different substitution/insertion rules.

\subsection{Two substitution rules $\R_1$ and $\R_2$}\label{ss:r12}
\begin{lemma}\label{L:rmin4}	
	Proposition \ref{P:rmin4} is true when $f_5^*$ is restricted between $\T_{5,3}\cap \A_n$ and the set $\B_1$ of ascent sequences $s$ from $\B$ (defined in Proposition \ref{P:rmin4}) where the non-zero integer $\min \M asc(s)$ does not appear after the rightmost $\Rmin(s)_{\rpos(s)}$.

\end{lemma}
Two substitution rules $\R_1$ and $\R_2$ are of central importance in the construction of this bijection, so we introduce them before proving Lemma \ref{L:rmin4}.

The key observation is that all substitutions (1)--(3) in $\R_1$ and (4)--(7) in $\R_2$ are reversible, by which all five Euler--Stirling statistics $\asc$, $\rep$, $\max$, $\zero$ and $\rmin$ are remained the same.  

For convenience, given an ascent sequence $s$, let $x_j=\Rmin(s)_j$ for all $0\le j<\rmin(s)$.

Rule $\CMcal{R}_1$:  For any ascent sequence $s$ such that for some $i$ the entry $x_i$ appears at least twice after the rightmost $x_{i-1}$, we will replace each non-rightmost $x_i$ by an $\M asc$ $m$ of $s$ as long as 
\begin{enumerate}[label=(\roman*)]
	\item $x_i<m$;
	\item $x_i$ is located after the leftmost $m$ and the rightmost $x_{i-1}$; 
	\item all entries between this $x_i$ and the rightmost $x_i$ are not equal to $m$. 
\end{enumerate}
This substitution procedure starts with the {\em first} $x_i$ that satisfies (ii), and then proceed with other non-rightmost $x_i$'s from left to right. 

Let $k_1$ and $k_2$ be the left and right neighbors of a given non-rightmost $x_i$ respectively, then ($k_1>x_i$ or $k_1=x_{i-1}$) and $m\ne k_2\ge x_i$, so there are only three possible scenarios:
\begin{enumerate}
	\item If one of the followings is true (see Figure \ref{F:b3}),
	\begin{itemize}
		\item $x_{i}<k_1<m$ and $x_i<k_2<m$;
		\item ($k_1\ge m$ or $k_1=x_{i-1}$) and ($k_2>m$ or $k_2=x_i$),
	\end{itemize}
	then replace the entry $x_i$ directly by $m$;	
	\item otherwise if  (see Figure \ref{F:b4}),
	\begin{itemize}
		\item $x_{i}<k_1<m$ and ($k_2>m$ or $k_2=x_i$);
	\end{itemize}
	then insert $m$ right before the leftmost entry that is to the left of $x_i$ and all entries between it and $k_1$ inclusive are greater than $x_i$ and smaller than $m$; afterwards remove $x_i$;
	\item otherwise (see Figure \ref{F:b4}),
	\begin{itemize}
		\item  ($k_1\ge m$ or $k_1=x_{i-1}$) and $x_i<k_2<m$, 
	\end{itemize}
	then insert $m$ right after the rightmost entry that is to the right of $x_i$ and all entries between it and $k_2$ inclusive are greater than $x_i$ and smaller than $m$; afterwards remove $x_i$.
\end{enumerate}	
\begin{figure}[ht]
	\centering
	\includegraphics[scale=0.75]{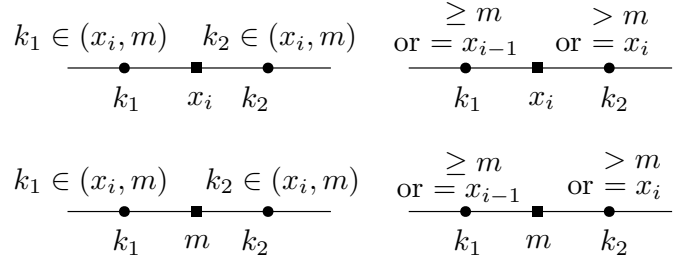}
	\caption{Substitution $(1)$ in rule $\R_1$. Here $x_i=\Rmin(s)_{i}$ and $i=\rpos(s)$.\label{F:b3}}
\end{figure}
\begin{figure}[ht]
	\centering
	\includegraphics[scale=0.85]{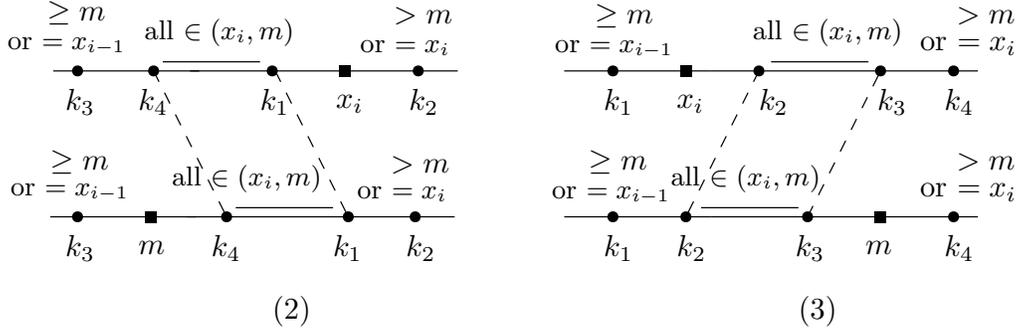}
	\caption{Substitutions $(2)$ in rule $\R_1$ and $(3)$ in rules $\R_1,\R_2$. Here $x_i=\Rmin(s)_{i}$ with $i=\rpos(s)$.\label{F:b4}}
\end{figure}

\begin{example}
	Given an ascent sequence $s=(0,1,2,0,1,4,1,2,1,1)$ where $x_1=1$ appears at least twice after the rightmost $x_{0}=0$ and $m=4$ is an $\M asc$, we will replace all non-rightmost $1$'s that are located after the leftmost $4$ by integers $4$ according to Rule $\R_1$:
	\begin{align*}
	s&=\,(0,1,2,0,1,4,{\bf 1},2,1,1) \quad \mbox{ by substitution } (3) \mbox{ of } \R_1,\\
	&\rightarrow (0,1,2,0,1,4,2,4,{\bf 1},1) \quad \mbox{ by substitution } (1) \mbox{ of } \R_1,\\
	&\rightarrow (0,1,2,0,1,4,2,4,4,1).
	\end{align*}
	It is easy to verify that $\asc,\rep,\zero,\max,\rmin$ are preserved under the rule $\R_1$.
\end{example}
Rule $\R_2$: in addition to the conditions (i), (ii) and (iii) listed in $\R_1$, here we require that
\begin{enumerate}[label=(\roman*)]
	\setcounter{enumi}{3}
	\item the two rightmost $x_i$ are not next to each other;
\end{enumerate}
Like $\R_1$, the procedure starts with the {\em first} $x_i$ that satisfies (ii), and proceed with other non-rightmost $x_i$'s from left to right.

Let $k_1$ and $k_2$ be the left and right neighbours of a given non-rightmost $x_i$ respectively, then ($k_1>x_i$ or $k_1=x_{i-1}$) and $m\ne k_2\ge x_i$, so there are four possible scenarios:
\begin{enumerate}
	\setcounter{enumi}{3}
	\item If $k_2=x_i$, then $k_2$ is not a right-to-left minimum (because of (iv)). Assume that $k_2$ is followed by exactly $k$ identical entries $x_i$ that are not right-to-left minima, then remove $k_2$ and these $k$ entries $x_i$, substitute $x_i$ by $m$ according to (5)--(7) below and finally add $(k+1)$ identical entries $m$ after the newly inserted $m$;
	\item otherwise $k_2\ne x_i$, if one of the followings is true (see Figure \ref{F:b5}),
	\begin{itemize}
		\item $x_i<k_1<m$ and $x_i<k_2<m$, 
		\item ($k_1>m$ or $k_1=x_{i-1}$) and $k_2>m$, 
	\end{itemize}
	then replace the entry $x_i$ by $m$;
	\item otherwise if (see Figure \ref{F:b5})
	\begin{itemize}
		\item $x_i<k_1\le m$ and $k_2>m$,
	\end{itemize}
	then insert $m$ right after the rightmost entry that is to the right of $x_i$ and all entries between it and $k_2$ inclusive are greater than $m$; afterwards remove $x_i$;
	\item otherwise, do $(3)$ of $\R_1$ (see Figure \ref{F:b4}).
\end{enumerate}

\begin{figure}[ht]
	\centering
	\includegraphics[scale=0.75]{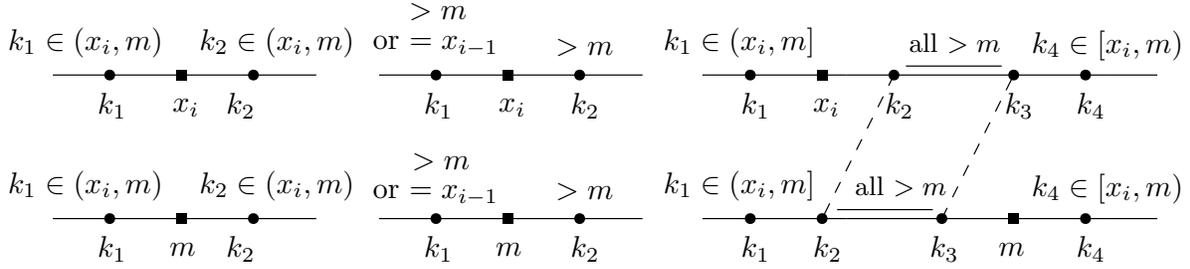}
	\caption{Substitution $(5)$ and $(6)$ in rule $\R_2$. Here $x_i=\Rmin(s)_{i}$ with $i=\rpos(s)$.\label{F:b5}}
\end{figure}
\begin{example}
	Given an ascent sequence $s=(0,1,2,0,1,4,4,1,5,2,1,3,1)$ where $x_1=1$ appears at least twice after the rightmost $x_{0}=0$ and $m=4$ is an $\M asc$, we will replace all non-rightmost $1$'s that are located after the leftmost $4$ by integers $4$ according to Rule $\R_2$:
	\begin{align*}
	s&=\,(0,1,2,0,1,4,4,{\bf 1},5,2,1,3,1) \quad \mbox{ by substitution } (6) \mbox{ of } \R_2,\\
	&\rightarrow (0,1,2,0,1,4,4,5,4,2,{\bf 1},3,1) \quad \mbox{ by substitution } (5) \mbox{ of } \R_2,\\
	&\rightarrow (0,1,2,0,1,4,4,5,4,2,4,3,1) .
	\end{align*}
	It is easy to verify that $\asc,\rep,\zero,\max,\rmin$ are preserved under the rule $\R_2$.
\end{example}

\begin{remark}
	The driving idea to define two different substitution rules $\R_1,\R_2$ is that {\sf Case} $1,2$ and {\sf Case} $3,4$ in the proof of Lemma \ref{L:rmin4} have to be treated differently.
\end{remark}

We are now in a position to complete the proof of Lemma \ref{L:rmin4}.

\begin{proof}
We start with showing the bijection
\begin{align*}
g:\T_{5,3}\cap \A_n\rightarrow \{s\in \A_{n-1}: \rpos(s)\ne 0\}.
\end{align*}	
For any ascent sequence $s\in \T_{5,3}\cap\A_n$ with $\rpos(s)=i$, replacing the rightmost $\Rmin(s)_i$ by $\sebr(s)$ and removing the last entry leads to an ascent sequence $s^*$ with $\rpos(s^*)=i+1$. Define $g(s)=s^*$ and clearly $g$ is invertible, so $g$ is a bijection. Similar to Lemma \ref{L:rmincase3}, it is straightforward to verify that $g$ transforms the quadruple 
\begin{align*}
(\asc,\rep,\max,\rmin,\rpos) \mbox{ to } (\asc+1,\rep,\max,\rmin,\rpos-1),
\end{align*}
and satisfies $\zero(g(s))=\zero(s)+\chi(\rpos(s)=0)$. If ${\sf Prm}_{\rpos(s)}\ne \max(s)+1$, then $\ealm(s)=\ealm(g(s))$; otherwise $\ealm(s)=\ealm(g(s))-1$.

We next define the map
\begin{align}\label{E:g51} 
g_{5,3}:\{s\in\A_{n-1}:\rpos(s)\ne 0\}\rightarrow \B_1
\end{align}
and then prove $g_{5,1}$ is a bijection so that
\begin{align*}
f_{5,3}:=g_{5,3}\circ g:\T_{5,3}\cap \A_n\rightarrow \B_1
\end{align*}
is the desired bijection $f_5^*$ when restricted to the subset $\T_{5,3}\cap \A_n$.

For any ascent sequence $s\in\A_{n-1}$ with $\rpos(s)=i\ne 0$, we discuss all possible scenarios and define the resulting sequence to be $g_{5,3}(s)$ in each case.

{\sf Case} $1$ (see Figure \ref{F:b6}): if the rightmost $\Rmin(s)_{i-1}$ is next to an $\Rmin(s)_{i}$ and there is at least one $\M asc$ between the first two $\Rmin(s)_{i}$ that are after the rightmost $\Rmin(s)_{i-1}$, let the smallest one be $m$, then 
\begin{itemize}
	\item insert $m+1$ right after the leftmost $m$;
	\item replace all entries $y$ after the inserted $m+1$ by $y+1$ if $y\ge m$;
	\item if there are only two $\Rmin(s)_{i}$ after the rightmost $\Rmin(s)_{i-1}$, then stop; otherwise, replace each $\Rmin(s)_{i}$ that appears between the leftmost $m$ and the rightmost $\Rmin(s)_i$ by an $m$ according to rule $\CMcal{R}_1$.
\end{itemize}
\begin{figure}[ht]
	\centering
	\includegraphics[scale=0.6]{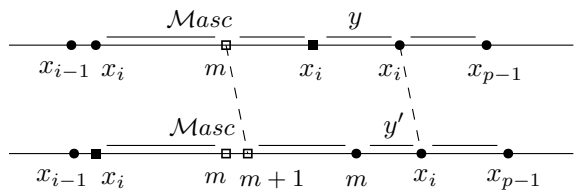}
	\caption{{\sf Case} $1$: the rightmost $x_{i-1}$ is next to $x_i$ and $m$ is the smallest $\M asc$ between the first two $x_i$'s that are after $x_{i-1}$. Here $x_i=\Rmin(s)_{i}$ with $i=\rpos(s)$ and $y'=y+1$ if $y\ge m$; otherwise $y'=y$. \label{F:b6}}
\end{figure}

{\sf Case} $2$ (see Figure \ref{F:b7}): if the rightmost $\Rmin(s)_{i-1}$ is next to $\Rmin(s)_{i}$ and {\em no} $\M asc$ appears between the first two $\Rmin(s)_{i}$ that are after the rightmost $\Rmin(s)_{i-1}$, let $m-1$ be the number of ascents from the beginning $s_1$ to the second $\Rmin(s)_{i}$ after the rightmost $\Rmin(s)_{i-1}$, then
\begin{itemize}
	\item insert $m$ right before the second $\Rmin(s)_{i}$ after the rightmost $\Rmin(s)_{i-1}$;
	\item replace any entry $y$ after the inserted $m$ by $y+1$ if $y\ge m$;
	\item if there are only two $\Rmin(s)_{i}$ after the rightmost $\Rmin(s)_{i-1}$, then stop; otherwise, replace each $\Rmin(s)_{i}$ that is between the leftmost $m$ and the rightmost $\Rmin(s)_i$ by an $m$ according to rule $\CMcal{R}_1$.
\end{itemize}
\begin{figure}[ht]
	\centering
	\includegraphics[scale=0.7]{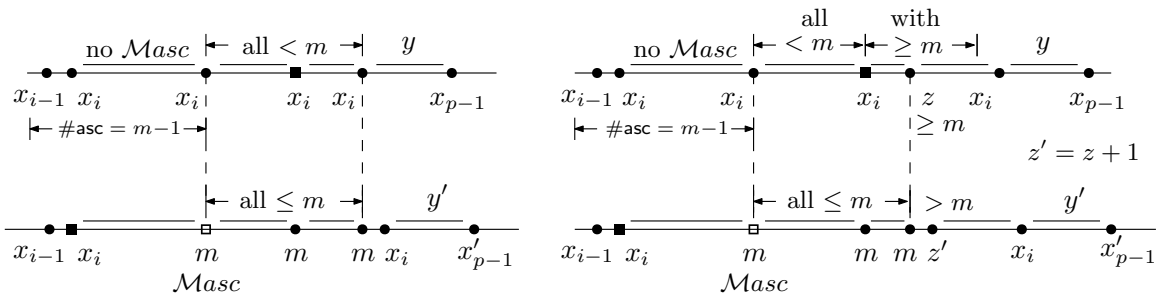}
	\caption{{\sf Case} $2$: the rightmost $x_{i-1}$ is next to $x_i$ and no $\M asc$ appears between the first two $x_i$'s that are after $x_{i-1}$. Here $x_i=\Rmin(s)_{i}$ with $i=\rpos(s)$ and $y'=y+1$ if $y\ge m$; otherwise $y'=y$.\label{F:b7}}
\end{figure}

{\sf Case} $3$ (see Figure \ref{F:b8}): if the rightmost $\Rmin(s)_{i-1}$ is {\em not} next to an $\Rmin(s)_{i}$ and the two rightmost $\Rmin(s)_{i}$ are not next to each other,  let $m-2$ be the number of ascents from the beginning $s_1$ to the first $\Rmin(s)_{i}$ after the rightmost $\Rmin(s)_{i-1}$, then 
\begin{itemize}
	\item insert $\Rmin(s)_{i}$ immediately after the rightmost $\Rmin(s)_{i-1}$;
	\item if the second $\Rmin(s)_i$ after the rightmost $\Rmin(s)_{i-1}$ is followed by exactly $k$ non-rightmost $\Rmin(s)_i$ ($k$ could be zero), then replace these $(k+1)$ identical entries $\Rmin(s)_i$ by $(k+1)$ identical $m$;
	\item replace all entries $y$ after the rightmost inserted $m$ by $y+1$ if $y\ge m$;
	\item substitute each $\Rmin(s)_{i}$ that is between the leftmost $m$ and the rightmost $\Rmin(s)_i$ by an $m$ according to rule $\CMcal{R}_2$.
\end{itemize}
\begin{figure}[ht]
	\centering
	\includegraphics[scale=0.62]{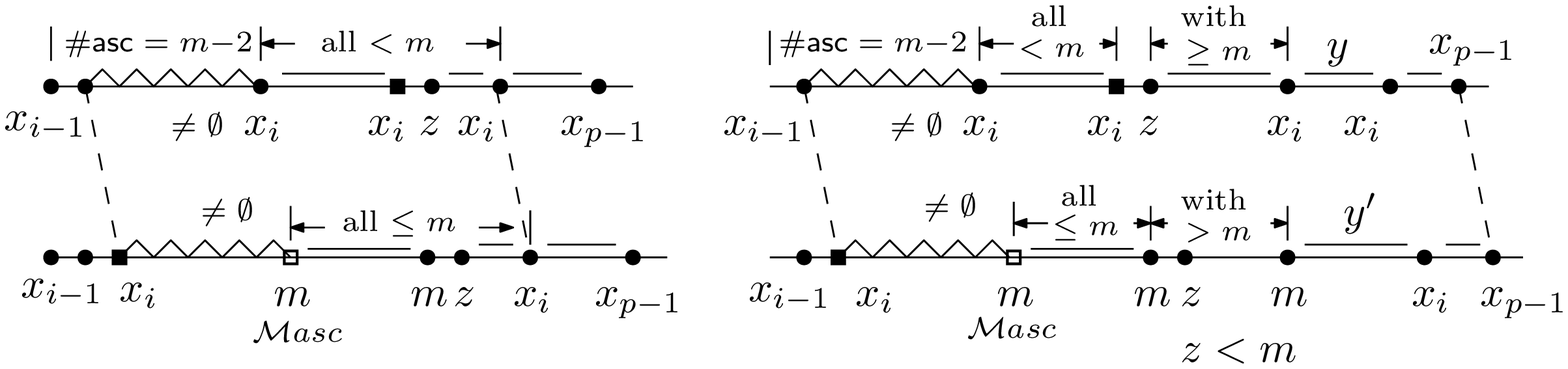}
	\caption{{\sf Case} $3$: the rightmost $x_{i-1}$ is not next to $x_i$ and the two rightmost $x_i$'s are not next to each other. Here $x_i=\Rmin(s)_{i}$ with $i=\rpos(s)$, $z<m$ and $y'=y+1$ if $y\ge m$; otherwise $y'=y$.\label{F:b8}}
\end{figure}

{\sf Case} $4$ (see Figure \ref{F:b9}): if the rightmost $\Rmin(s)_{i-1}$ is not next to an $\Rmin(s)_{i}$ and the two rightmost $\Rmin(s)_{i}$ are next to each other,  let $m-2$ be the number of ascents from the beginning to the rightmost $\Rmin(s)_{i-1}$, then, assuming that exactly $(k+1)$ rightmost $\Rmin(s)_{i}$ are next to each other ($k\ge 1$), we
\begin{itemize}
	\item remove the $k$ rightmost $\Rmin(s)_{i}$;
	\item insert two integers $\Rmin(s)_{i}m$ immediately after the rightmost $\Rmin(s)_{i-1}$;
	\item replace all entries $y$ after the inserted $m$ by $y+1$ if $y\ge m$;
	\item substitute each non-rightmost $\Rmin(s)_{i}$ that are between the leftmost $m$ and the rightmost $\Rmin(s)_{i}$ by an $m$ according to rule $\CMcal{R}_2$;
	\item insert $(k-1)$ $m$'s immediately after the leftmost $m$.
\end{itemize}
\begin{figure}[ht]
	\centering
	\includegraphics[scale=0.77]{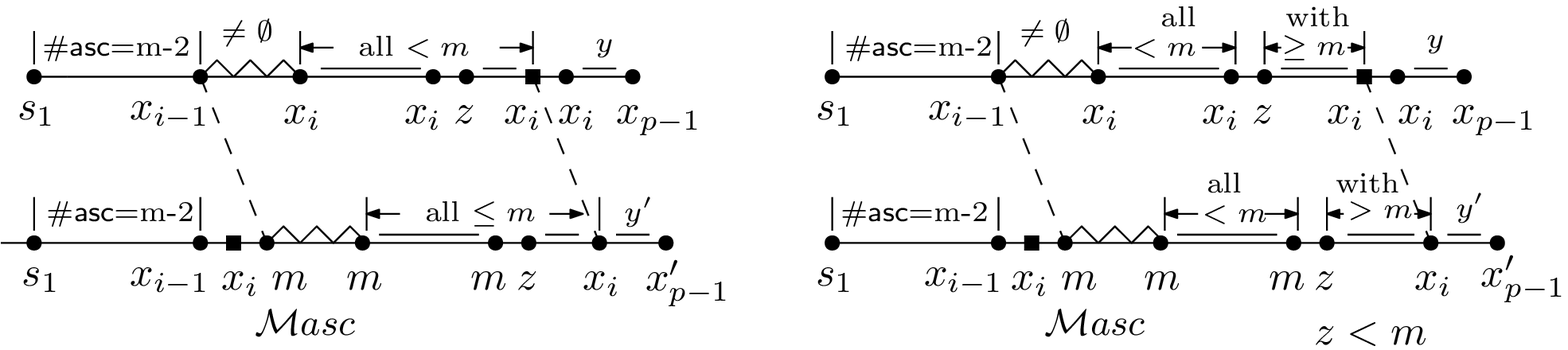}
	\caption{{\sf Case} $4$: the rightmost $x_{i-1}$ is not next to $x_i$ and the two rightmost $x_i$'s are next to each other. Here $x_i=\Rmin(s)_{i}$ with $i=\rpos(s)$, $z<m$ and $y'=y+1$ if $y\ge m$; otherwise $y'=y$.\label{F:b9}}
\end{figure}
By the construction of $g_{5,3}(s)$ (see (\ref{E:g51})), one can readily see that $g_{5,3}(s)\in\B_1$. 
It remains to show that $g_{5,3}$ is a bijection. 

For any ascent sequence $\hat{s}\in \B_1$ with $\rpos(\hat{s})=i\ne 0$ and $\min \M asc(\hat{s})=\hat{m}$, if all entries between the two rightmost $\Rmin(\hat{s})_i$ are less than or equal to $\hat{m}$, then $\hat{s}$ is produced from 
\begin{itemize}
	\item {\sf Case} $2$ if the last $\hat{m}$ is next to the rightmost $\Rmin(\hat{s})_i$ (see the left one of Figure \ref{F:b7});
	\item {\sf Case} $3$ otherwise if the first $\hat{m}$ is not next to $\Rmin(\hat{s})_i$ (see the left one of Figure \ref{F:b8});
	\item {\sf Case} $4$ otherwise (see the left one of Figure \ref{F:b9}).
\end{itemize}
If there exists an entry that is greater than $\hat{m}$ and appears between the two rightmost $\Rmin(\hat{s})_i$, then $\hat{s}$ comes from 
\begin{itemize}
	\item {\sf Case} $1$ if the first $\hat{m}$ is followed by $\hat{m}+1$ (see Figure \ref{F:b6});
	\item {\sf Case} $2$ otherwise if the first entry that is greater than $\hat{m}$ appears immediately after a non-leftmost $\hat{m}$ (see the right one of Figure \ref{F:b7});
	\item {\sf Case} $3$ otherwise if the first $\hat{m}$ is not next to $\Rmin(\hat{s})_i$ (see the right one of Figure \ref{F:b8});
	\item {\sf Case} $4$ otherwise (see the right one of Figure \ref{F:b9}).
\end{itemize}
This implies that $g_{5,3}$ is surjective. Since all steps in all cases including the substitution rules $\R_1$ and $\R_2$ are reversible, the map $g_{5,3}$ is therefore injective. In consequence, $g_{5,3}$ is a bijection, implying the composition $f_{5,3}=g_{5,3}\circ g$ is the desired bijection $f_5^*$ when restricted to the set $\T_{5,3}\cap\A_n$.

Regarding the statistics, the bijection $f_{5,3}$ sends $(\asc,\rep,\rmin,\rpos)$ to $(\asc-1,\rep,\rmin,\rpos)$.  Only when $\rpos(s)=0$, $\zero(f_{5,3}(s))=\zero(s)+1$. In analogy to Lemma \ref{L:rmincase3}, one can examine the change of statistics $\max$ and $\ealm$.
\end{proof}

\begin{example}
	For $s=(0,1,2,0,1,3,{2},5,5,{2},7,3,{\bf 1},3,8)\in\T_{5,3}\cap \A_{15}$, then by applying bijection $g$, we obtain $g(s)=(0,1,2,0,1,3,{2},5,5,{2},7,3,{\bf 2},3)$ which belongs to {\sf case 3}. According to the construction of $g_{5,3}$ for {\sf case} 3, we have $m=6$ and 
	\begin{align*}
	g(s)&=\,(0,1,2,0,1,3,{\bf 2},5,5,{\bf 2},7,3,{\bf 2},3)\\
	&\rightarrow (0,1,2,0,1,{\bf 2},3,{\bf 2},5,5,{\bf 2},8,3,{\bf 2},3)\\
	&\rightarrow (0,1,2,0,1,{\bf 2},3,6,5,5,{\bf 2},8,3,{\bf 2},3)\\
	&\rightarrow (0,1,2,0,1,{\bf 2},3,6,5,5,8,{6},3,{\bf 2},3)=f_{5,3}(s)=f_5^*(s).
	\end{align*}
\end{example}

We next turn to construct the bijection for the subset $\T_{5,4}$ where two insertion rules $\R_3,\R_4$ that modifies the set of right-to-left minima play an essential role. 

\subsection{Two insertion rules $\R_3$ and $\R_4$}\label{ss:r34}

\begin{lemma}\label{L:rmin5}	
	Proposition \ref{P:rmin4} is true when $f_5^*$ is restricted between $\T_{5,4}\cap \A_n$ and the set $\B-\B_1$ of ascent sequences $s$ from $\B$ (defined in Proposition \ref{P:rmin4}) where the non-zero integer $\min \M asc(s)$ also appears after the rightmost $\Rmin(s)_{\rpos(s)}$.
\end{lemma}
We prove Lemma \ref{L:rmin5} right after the rules $\R_3$ and $\R_4$ are defined.

Rule $\R_3$: For any ascent sequence $s$, let $m$ be an $\M asc$ of $s$ that appears only once and it is not a right-to-left minimum, set 
$$k=\max\{l: \Rmin(s)_l\le m-1\},$$
we will insert an $m$ to $s$ so that $m$ becomes a new right-to-left minimum.  
\begin{itemize}
	\item 	
	If $k=\rmin(s)-1$, i.e., the last right-to-left minimum $\Rmin(s)_{k}$ (or equivalently the last entry)
	is smaller than $m$, then we add $m$ at the end of $s$; otherwise, we replace the rightmost $\Rmin(s)_{k+1}$ by $m$, replace the rightmost $\Rmin(s)_{r+1}$ by $\Rmin(s)_r$ for $k+1\le r\le \rmin(s)-2$ and add $\Rmin(s)_{\rmin(s)-1}$ at the end.    
	\begin{figure}[ht]
		\centering
		\includegraphics[scale=0.66]{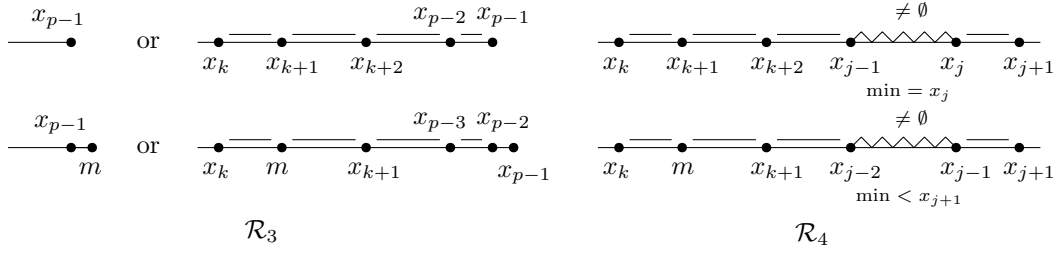}
		\caption{The insertion rules $\R_3$ and $\R_4$ where $x_i=\Rmin(s)_{i}$, $j=\rpos(s)$ and $k$ is the maximal index such that $x_k\le m-1$.\label{F:c0}}
	\end{figure}
\end{itemize}

Rule $\R_4$: in addition to the conditions of Rule $\R_3$, here we also required that $k<\rpos(s)$. We insert an $m$ and remove the rightmost $\Rmin(s)_{\rpos(s)}$ so that $m$ is a new right-to-left minimum. This is achieved by replacing the rightmost $\Rmin(s)_{k+1}$ by $m$, replacing the rightmost $\Rmin(s)_{r+1}$ by $\Rmin(s)_r$ for $k+1\le r\le \rpos(s)-1$.


We are now ready to prove Lemma \ref{L:rmin5}.

\begin{proof}
For any ascent sequence $s\in \T_{5,4}$, we distinguish three cases according to the location of the first $\M asc$ after the rightmost $\Rmin(s)_{\rpos(s)}$. For the first two cases, the map $s\mapsto f_{5,4}(s)$ is explicitly defined, based on which the map $s\mapsto f_{5,4}(s)$ for the remaining case is recursively constructed.

{\sf Case} $1$ (see Figure \ref{F:c1}): if the first $\M asc$ after the rightmost $\Rmin(s)_{\rpos(s)}$ is a right-to-left minimum, we then implement the following {\sf Step} $1$ on the pair $(s,\rpos(s))$ to construct a new sequence $f_{5,4}(s)\in \B-\B_1$.

{\sf Step} $1$ (see Figure \ref{F:c1}): 

For any pair $(s,u)$ where $s\in\T_{5,4}$ and $u\le \rpos(s)$,  assume that the rightmost $\Rmin(s)_j$ is the first $\M asc$ after the rightmost $\Rmin(s)_{\rpos(s)}$, then
\begin{itemize}
	\item remove all entries after the rightmost $\Rmin(s)_{j-1}$;
	\item remove the rightmost $\Rmin(s)_{\rpos(s)}$;
	\item if $u=\rpos(s)$, apply the bijection $g_{5,3}$ (see (\ref{E:g51}));
	\item let $m$ equal the minimal $\M asc$ between the two rightmost $\Rmin(s)_{u+1}$, and insert $m$ according to rule $\R_3$;
	\item replace each non-rightmost $\Rmin(s)_t$ entries that is located after $\Rmin(s)_{t-1}$ by an $m$ according to rule $\R_1$ if $u+2\le t\le \max\{l:\Rmin(s)_l<m\}$; 
	\item add $(\rmin(s)-j-1)$ $\M asc$'s at the end.
\end{itemize}
Define $f_{5,4}(s)$ to be the resulting sequence after applying {\sf Step} $1$ to the pair $(s,\rpos(s))$.

\begin{figure}[ht]
	\centering
	\includegraphics[scale=0.7]{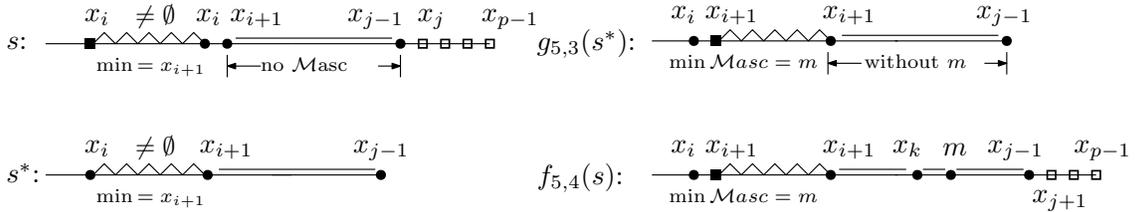}
	\caption{ The construction of $s\mapsto f_{5,4}(s)$ for {\sf case} $1$ when the first $\M asc$ after the rightmost $x_i$ is a right-to-left minimum. Here $x_i=\Rmin(s)_{i}$ and $i=\rpos(s)$. \label{F:c1}}
\end{figure}

{\sf Case} $2$ (see Figure \ref{F:c21}): if the first $\M asc$ after the rightmost $\Rmin(s)_{\rpos(s)}$ appears exclusively between two right-to-left minima, then we implement {\sf Step} $2$ on the pair $(s,\rpos(s))$ to construct a new sequence $f_{5,4}(s)\in \B-\B_1$.

{\sf Step} $2$ (see Figure \ref{F:c21}): 

For any pair $(s,u)$ where $s\in \T_{5,4}$ and $ u\le \rpos(s)$, assume that the first $\M asc$ after the rightmost $\Rmin(s)_{\rpos(s)}$ appears exclusively between the rightmost $\Rmin(s)_{j-1}$ and $\Rmin(s)_j$, then
\begin{itemize}
	\item remove the rightmost $\Rmin(s)_{\rpos(s)}$;
	\item add $\Rmin(s)_j$ right after the rightmost $\Rmin(s)_{j-1}$;
	\item if $u=\rpos(s)$, then separate the sequence right after the rightmost $\Rmin(s)_{j-1}$; apply $g_{5,3}$ (see \eqref{E:g51}) to the left part; let $m$ be the minimal $\M asc$ between the two rightmost $\Rmin(s)_{u+1}$; replace all entries $y$ from the right part by $y+1$ if $y\ge m$; afterwards put these two parts back together (see Figure \ref{F:c21});
	\item apply $g_{5,3}^{-1}$ to the entire sequence;
	\item Let $m$ be the minimal $\M asc$ between the two rightmost $\Rmin(s)_{u+1}$ and $k:=\max\{l:\Rmin(s)_l<m\}$. If $k<j$, then insert $m$ according to rule $\R_4$;
	\item replace every non-rightmost $\Rmin(s)_t$ that is located after the rightmost $\Rmin(s)_{t-1}$ by an $m$ according to rule $\R_1$ if $u+2\le t\le k$. See Figure \ref{F:c21}. 	
\end{itemize}
\begin{figure}[ht]
	\centering
	\includegraphics[scale=0.50]{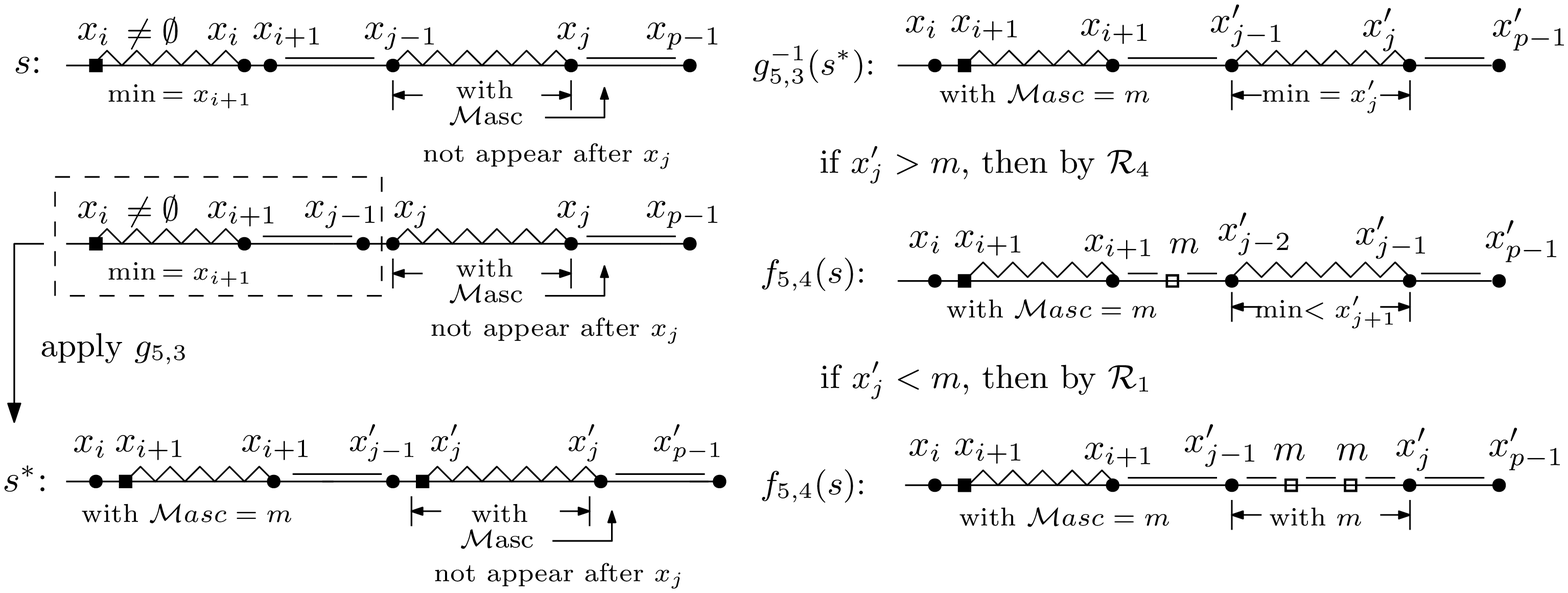}
	\caption{The construction $s\mapsto f_{5,4}(s)$ for {\sf Case} $2$. Here $x_i=\Rmin(s)_{i}$ with $i=\rpos(s)$ and $y'=y+1$ if $y\ge m$; otherwise $y=y'$. \label{F:c21}}
\end{figure}
Define $f_{5,4}(s)$ to be the resulting sequence after applying {\sf Step} $2$ to the pair $(s,\rpos(s))$. We next show the image sets of $f_{5,4}$ for {\sf Case} $1$ and {\sf Case} $2$ are disjoint.

For any $\hat{s}\in\B-\B_1$ with $m=\min \M asc(\hat{s})$, we divide $\B-\B_1$ into two disjoint subsets $\C_1$ and $\C_2$: $\C_1$ contains all ascent sequence $\hat{s}\in\B-\B_1$ satisfying the following conditions:
\begin{itemize}
	\item $m$ is a right-to-left minimum, say the $(k+1)$th right-to-left minimum;
	\item either two rightmost $\Rmin(\hat{s})_{t-1}$ and $\Rmin(\hat{s})_{t}$ are next to each other or the minimal entry in between is greater than or equal to $\Rmin(\hat{s})_{t+1}$ for all $k+1\le t\le \rmin(\hat{s})-1$.
\end{itemize}
Let $\C_2:=\B-\B_1-\C_1$. By the construction of $f_{5,4}(s)$ in {\sf Case} $1$--$2$, it is clear that the image set of $f_{5,4}(s)$ for {\sf Case} $1$ is a subset of $\C_1$, while the one for {\sf Case} $2$ is a subset of $\C_2$. Together with the fact that all steps are reversible, it follows that $f_{5,4}$ is injective for these two cases, from which we will recursively define the map $f_{5,4}$ for the remaining case.

First for any ascent sequence $s\in \T_{5,4}$ with $\rmin(s)-\rpos(s)=2$, $s$ must belong to {\sf Case} $1$ or $2$. Since the image set of $f_{5,4}(s)$ when $\rmin(s)-\rpos(s)=2$ is exactly $\C_1\dot\cup\,\C_2$ and $f_{5,4}$ is injective for these two cases, $f_{5,4}$ is a bijection when $\rmin(s)-\rpos(s)=2$.

Next assuming that there is a bijection $f_{5,4}: \T_{5,4}\cap \A_n\rightarrow \C_1\dot\cup\,\C_2$ for all ascent sequences $s$ with $\rmin(s)-\rpos(s)\le N$, we will construct the map $f_{5,4}$ for the ones with $\rmin(s)-\rpos(s)=N+1$ and prove it is a bijection. 

For any ascent sequence $s\in \T_{5,4}$ with $\rmin(s)-\rpos(s)=N+1$, if $s$ belongs to {\sf case} $1$ or $2$, then $f_{5,4}(s)$ is already given and we stop; otherwise $s$ must belong to the following case and a new sequence $f_{5,4}(s)\in \C_1\dot\cup\,\C_2$ will be defined.

{\sf Case} $3$ (see Figure \ref{F:c31}): if the first $\M asc$ after the rightmost $\Rmin(s)_{\rpos(s)}$ appears not only between two right-to-left minima, but also afterwards, then we implement the following step on the pair $(s,\rpos(s))$ to produce a new sequence $f_{5,4}(s)$.

{\sf Step} $3$ (see Figure \ref{F:c31}): 

For any pair $(s,u)$ where $s\in \T_{5,4}$ and $u\le \rpos(s)$, assume that the first $\M asc$ after the rightmost $\Rmin(s)_{\rpos(s)}$ appears between the rightmost $\Rmin(s)_{j-1}$ and $\Rmin(s)_{j}$, as well as after the rightmost $\Rmin(s)_{j}$, then
\begin{itemize}
	\item remove the rightmost $\Rmin(s)_{\rpos(s)}$;
	\item add $\Rmin(s)_j$ right after the rightmost $\Rmin(s)_{j-1}$;
	\item if $u=\rpos(s)$, then separate the sequence right after the rightmost $\Rmin(s)_{j-1}$; apply $g_{5,3}$ (see \eqref{E:g51}) to the left part; let $m$ be the minimal $\M asc$ between the two rightmost $\Rmin(s)_{u+1}$; replace all entries $y$ from the right part by $y+1$ if $y\ge m$; afterwards put these two parts back together (see Figure \ref{F:c31});
	\item apply $f_{5,4}^{-1}$ according to induction hypothesis and let $s^\bullet$ denote the resulting sequence;
	\item if $s^\bullet$ belongs to {\sf Case} $1$, do {\sf Step} $1$ on the pair $(s^\bullet,\rpos(s))$ and then stop;
	\item if $s^\bullet$ belongs to {\sf Case} $2$, do {\sf Step} $2$ on the pair $(s^\bullet,\rpos(s))$ and then stop;
	\item otherwise repeat {\sf Step} $3$ on the pair $(s^\bullet,\rpos(s))$.
\end{itemize}
Define $f_{5,4}(s)$ to be the resulting sequence after applying {\sf Step} $3$ to the pair $(s,\rpos(s))$.
\begin{figure}[ht]
	\centering
	\includegraphics[scale=0.54]{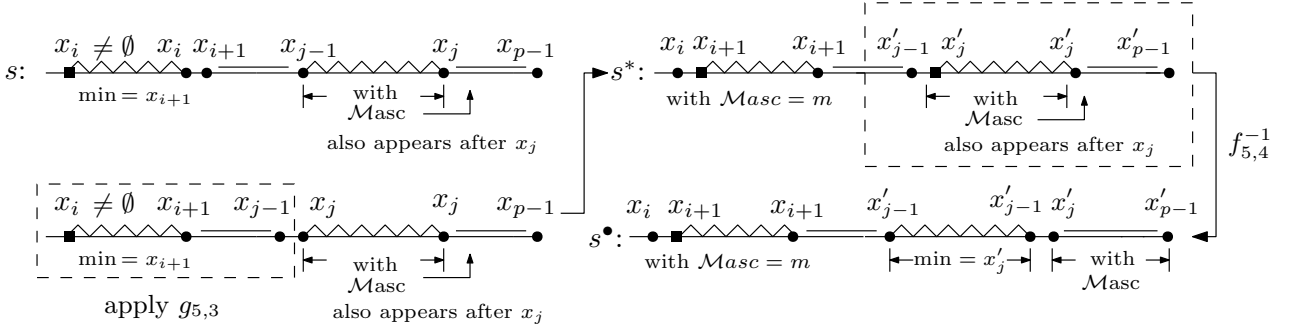}
	\caption{The construction of $s\mapsto s^{\bullet}$ for {\sf Case} $3$ and we repeat {\sf Step}s $1$--$3$ on the pair $(s^\bullet,\rpos(s))$ where $\rpos(s)=i$. \label{F:c31}}
\end{figure}

By the construction of $f_{5,4}$ for {\sf Case} $3$, it is clear that $f_{5,4}(s)\in \C_1\dot\cup\,\C_2$. According to the induction hypothesis, it remains to prove that the map $f_{5,4}$ is a bijection for all $s\in \T_{5,4}$ such that $\rmin(s)-\rpos(s)=N+1$. 

For any $\hat{s}\in \B-\B_1=\C_1\dot\cup\,\C_2$ with $\min \M asc(\hat{s})=m$ and $\rmin(\hat{s})-\rpos(\hat{s})=N$, then the sequence  $\hat{s}$ is generated from
\begin{itemize}
	\item {\sf Step} $1$ if $\hat{s}\in\C_1$;
	\item {\sf Step} $2$ if $\hat{s}\in\C_2$;
\end{itemize}
Since all {\sf Step}s $1$--$3$ including rules $\R_1,\R_3,\R_4$ are recursively reversible, we apply {\sf Step} $i$ in reverse order to $\hat{s}$ if $\hat{s}\in \C_i$ and obtain a pair $(s^{\bullet},u)$ with $s^{\bullet}\in\T_{5,4}$ and $u=\rpos(\hat{s})-1$. If $u=\rpos(s^{\bullet})$, then we stop and $s^{\bullet}=f_{5,4}^{-1}(\hat{s})$ with $\rmin(s^{\bullet})-\rpos(s^{\bullet})=N+1$; otherwise $u<\rpos(s^{\bullet})$, we implement {\sf Step} $3$ in reverse order on $s^{\bullet}$ (allowed by induction hypothesis) until a pair $(s,\rpos(s))$ is produced with $s=f_{5,4}^{-1}(\hat{s})$ satisfying $\rmin(s)-\rpos(s)=N+1$. This implies that the map $f_{5,4}$ is surjective and injective, that is $f_{5,4}$ is the desired bijection $f_5^*$ (defined in Proposition \ref{P:rmin4}) when restricted to the set $\T_{5,4}\cap\A_n$.
\end{proof}

\begin{example}\label{Eg:4}
	Given $s^\bullet=(0,1,2,0,1,2,5,{\bf 2},{3},{\bf 2},{3},8,8,4)\in\T_{5,4}$ with $\rpos(s^{\bullet})=2$, $s^{\bullet}$ belongs to {\sf Case} $2$, so we implement {\sf Step} $2$ on the pair $(s^{\bullet},2)$ as follows:
	\begin{align*}
	s^{\bullet}&\,=(0,1,2,0,1,2,5,{\bf 2},{3},{\bf 2},{3},8,8,4)\\
	&\rightarrow (0,1,2,0,1,2,5,{2},{3},{3},{\bf 4},8,8,{\bf 4}),
	\end{align*}
	then apply the bijection $g_{5,3}$ from (\ref{E:g51}) to the prefix $(0,1,2,0,1,2,5,{2},{3},{3})$ and obtain
	\begin{align*}
	g_{5,3}((0,1,2,0,1,2,5,{2},{\bf 3},{\bf 3}))=(0,1,2,0,1,2,5,{2},{\bf 3},7,{\bf 3}).
	\end{align*}
	Add the subsequence $(4,9,9,4)$ at the end, where $(4,9,9,4)$ comes from increasing each entry of $(4,8,8,4)$ by one if it is larger than or equal to $m=7$. This leads to the ascent sequence
	\begin{align*}
	(0,1,2,0,1,2,5,{2},{3},7,{3},{\bf 4},9,9,{\bf 4})\in\B_1.
	\end{align*}
	Next after applying the inverse bijection $g_{5,3}^{-1}$, it becomes
	\begin{align*}
	g_{5,3}^{-1}((0,1,2,0,1,2,5,{2},{3},7,{3},{\bf 4},9,9,{\bf 4}))=(0,1,2,0,1,2,5,{2},{3},7,{3},{4},{\bf 4},{\bf 4}).
	\end{align*}
	Finally substitute non-rightmost entries $4$ by $7$ according to $\R_1$ and
	\begin{align*}
	f_5^*(s^\bullet)=f_{5,4}(s^{\bullet})=(0,1,2,0,1,2,5,{2},{\bf 3},7,{\bf 3},{7},{7},{4}).
	\end{align*}
\end{example}

\begin{example}
	Given $s=(0,1,2,0,{\bf 1},2,{\bf 1},2,6,3,6,6,4)\in\T_{5,4}$ with  $\rpos(s)=1$ and $\rmin(s)=5$. Since $s$ belongs to {\sf Case} $3$, we implement {\sf Step} $3$ on the pair $(s,1)$ as follows:	
	\begin{align*}
	s&\,=(0,1,2,0,{\bf 1},2,{\bf 1},2,6,3,6,6,4)\\
	&\rightarrow (0,1,2,0,{1},{2},{2},{\bf 3},6,{\bf 3},6,6,4);
	\end{align*}
	then apply the bijection $g_{5,3}$ from (\ref{E:g51}) to the subsequence $(0,1,2,0,1,2,2)$, yielding
	\begin{align*}
	f_{5,1}((0,1,2,0,1,2,2))=(0,1,2,0,1,2,5,2);
	\end{align*}
	attach the subsequence $({\bf 3},7,{\bf 3},7,7,4)$ at the end, where $({\bf 3},7,{\bf 3},7,7,4)$ comes from replacing each entry $y$ of $({\bf 3},6,{\bf 3},6,6,4)$ by $y+1$ if $y\ge m=5$. Now the ascent sequence becomes
	\begin{align*}
	(0,1,2,0,1,2,5,2,{\bf 3},7,{\bf 3},7,7,4)\in \B-\B_1;
	\end{align*}
	next apply the bijection $f_{5,4}^{-1}$ (by induction hypothesis) and it is known from Example \ref{Eg:4} that 
	\begin{align*}
	f_{5,4}^{-1}(0,1,2,0,1,2,5,2,{\bf 3},7,{\bf 3},7,7,4)=(0,1,2,0,1,2,5,{\bf 2},{3},{\bf 2},{3},8,8,4)=s^{\bullet}.
	\end{align*}
	Since $s^\bullet$ belongs to {\sf Case} $2$, we implement {\sf Step} $2$ on the pair $(s^\bullet,1)$ and get
	\begin{align*}
	(0,1,2,0,1,2,5,{\bf 2},{3},{\bf 2},{3},8,8,{\bf 4})
	&\rightarrow (0,1,2,0,1,2,5,{2},{3},{3},{\bf 4},8,8,{\bf 4})\\
	&\rightarrow (0,1,2,0,1,2,5,{2},{3},{3},{\bf 4},4,{\bf 4})\\
	&\rightarrow (0,1,2,0,1,2,5,{2},{5},{3},{\bf 5},5,{\bf 4})=f_{5,4}(s)=f_5^*(s).
	\end{align*}
\end{example}

Let $\T_5=\dot\cup_{i=1}^4\T_{5,i}$, we now redivide the subset $\T_5$ according to the change of statistics and summarize the bijections on the subsets $\T_{5,i}$, $1\le i\le 4$ in Proposition \ref{L:f5}.

Let $\M_{5,1}$ be the set of ascent sequences $s\in\T_{5}$ whose the second rightmost entry $\Rmin(s)_{\rpos(s)}$ is not a maximal of $s$ or $s\in\T_{5,2}-\M_{5,2}$, where 
\begin{align*}
\M_{5,2}:=\{s\in \T_{5,2}: \Prm(s)_{\rpos(s)}=\max(s)+1\}.
\end{align*}
Let $\M_{5,3}$ denote the set of ascent sequences $s\in\T_{5,3}\dot\cup\,\T_{5,4}$ whose second rightmost entry $\Rmin(s)_{\rpos(s)}$ is a maximal of $s$. By definition, $\T_{5,1}\subseteq\M_{5,1}$ and  $\T_5=\M_{5,1}\dot\cup\,\M_{5,2}\dot\cup\,\M_{5,3}$.

\begin{proposition}\label{L:f5}
	There is a bijection 
	\begin{align*}
	f_5:\T_5\cap \A_n \rightarrow (\T_3\dot \cup \T_4 \dot \cup \T_5) \cap \A_n
	\end{align*}
	that satisfies $\zero(s)=\zero(f_{5}(s))+\chi(\rpos(s)=0)$ and transforms
	\begin{align*}
	(\asc,&\rep,\rmin,\rpos) \mbox{ to } (\asc,\rep,\rmin,\rpos-1);\\
	&(\max,\ealm) \mbox{ to } (\max,\ealm), \mbox{ if } s\in \M_{5,1},\\
	&(\max,\ealm) \mbox{ to } (\max,\ealm-1), \mbox{ if } s\in \M_{5,2},\\
	&(\max,\ealm) \mbox{ to } (\max-1,\ealm-1), \mbox{ if } s\in \M_{5,3}.
	\end{align*}
\end{proposition}
\begin{proof}
	For $1\le i\le 4$ and any ascent sequence $s\in \T_{5,i}$, set $f_5(s)=f_{5,i}(s)$ where $f_{5,3},f_{5,4}$ are bijections $f_5^*$ when restricted to the subsets $\T_{5,3},\T_{5,4}$ respectively. It is not hard to see that $f_5$ satisfies all desired properties after combining Lemma \ref{L:f53}, \ref{L:f54} and Proposition \ref{P:rmin4} (Lemma \ref{L:rmin4} and \ref{L:rmin5}).
\end{proof}

\section{Proof of Theorem \ref{T:5tuple}}\label{Sec:7tuple}
The purpose of this section is to complete the proof of Theorem \ref{T:5tuple}. We start with reviewing the decomposition of ascent sequences from \cite{fjlyz}, which is the last ingredient in this proof.

The decomposition of ascent sequences from \cite{fjlyz} is inspired by the statistic $\ealm$ (see Definition \ref{D:ealm}). In order to show it is parallel to the one in Section \ref{S:thm3},  we formulate it and the follow-up lemmas in a slightly different manner from \cite{fjlyz}. 

The set $\A^*$ of ascent sequences $s$ except $s=(0,1,2,\ldots,|s|-1)$ is partitioned into the following disjoint subsets:
\begin{align*}
\CMcal{D}_{1}&:=\{s\in\A^*: \vert s\vert=\max(s)+1\},\\
\CMcal{D}_{2}&:=\{s\in\A^*-\CMcal{D}_{1}: s_{\max(s)+2}\le \ealm(s)\},\\
\CMcal{D}_{3}&:=\{s\in\A^*-\CMcal{D}_{1}: s_{\max(s)+2}=\ealm(s)+1,\max(s)\notin\{s_i: \max(s)+2\leq i\leq |s|\}\},\\
\CMcal{D}_{4}&:=\{s\in\A^*-\CMcal{D}_{1}: s_{\max(s)+2}=\ealm(s)+1,\max(s)\in\{s_i: \max(s)+2\leq i\leq |s|\}\},\\
\CMcal{D}_{5}&:=\{s\in\A^*-\CMcal{D}_{1}: s_{\max(s)+2}\ge \ealm(s)+2\}.
\end{align*}
We provide the following bijections explicitly in the proofs, but omit other details as they are very straightforward.
\begin{lemma}[Lemma $8$ of \cite{fjlyz}]\label{L:maxcase2}
	There is a bijection
	$$h_2:\D_2\cap \A_n\rightarrow \{(i,s):s\in\A^*\cap \A_{n-1}, \ealm(s)\le i< \max(s)\}$$ 
	that sends $s$ to a pair $h_2(s)=(\ealm(s),s^*)$ satisfying
	\begin{equation*}
	(\asc,\rmin,\rpos,\max)s=(\asc,\rmin,\rpos,\max)s^*,
	\end{equation*}
\begin{equation*}
\zero(s)=\zero(s^*)+\chi(\ealm(s)=0)\quad\;
\text{and}\quad\;\rep(s)=\rep(s^*)+1.
\end{equation*}
\end{lemma}
\begin{proof}
	For any $s\in\D_2$, remove the entry $\ealm(s)$ at the $(\max(s)+1)$-th position of $s$. Let the resulting sequence be $s^*$ and define $h_2(s)=(\ealm(s),s^*)$.
\end{proof}
Let $\CMcal{P}_2$ be the set of ascent sequences $s\in\A^*$ such that the integer $\max(s)-1$ appears exactly once in $s$. Denote by $\CMcal{P}_2^c$ the complement of $\CMcal{P}_2$ in $\A^*$.
\begin{lemma}[Lemma 10 of \cite{fjlyz}]\label{L:p2}
	There is a bijection $$\phi_2:\A_n \cap \CMcal{P}_2\rightarrow \A_{n-1}\cap \A^*$$ 
	that transforms the septuple 
	$$(\asc,\rep,\zero,\max,\ealm,\rmin,\rpos) \mbox{ to } (\asc+1,\rep,\zero,\max+1,\ealm,\rmin,\rpos).$$
\end{lemma}
\begin{proof}
	For any $s\in\CMcal{P}_2$, remove the unique entry $\max(s)-1$ and replace all entries $y$ by $y-1$ if $y\ge \max(s)$. Let $\phi_2(s)$ be the resulting sequence.
\end{proof}
\begin{lemma}[Lemma $9$ of \cite{fjlyz}]\label{L:maxcase5}
	There is a bijection 
	\begin{align*}
	h_3:\D_3 \cap \A_n \rightarrow \{s\in \A_{n}\cap \CMcal{P}_2:\ealm(s)\ne 0\}
	\end{align*}
	that transforms the quintuple
	\begin{equation*}
      (\asc,\rep,\rmin,\max,\ealm) \mbox{ to } (\asc,\rep+1,\rmin,\max-1,\ealm-1),
	\end{equation*}
and satisfies
\begin{align*}\zero(s)&=\zero(h_3(s))+\chi(\ealm(s)=0),\\ 
	\rpos(s)&=\rpos(h_3(s))-\chi(\Prm(s)_{\rpos(s)}=\max(s)+1).
	\end{align*}
\end{lemma}
\begin{proof}
	For any $s\in\D_3$, replace the entry $\ealm(s)$ on the $(\max(s)+1)$-th position by $\max(s)$. Define $h_3(s)$ as the resulting sequence.
\end{proof}

\begin{lemma}[Lemma $11$ of \cite{fjlyz}]\label{L:maxcase4}
	There is a bijection 
	\begin{align*}
	h_4:\D_4\cap \A_n\rightarrow &\{s\in \A_n\cap \CMcal{P}_2^c: \ealm(s)\ne 0\}
	\end{align*}
	that transforms the quintuple 
	\begin{align*}
	(\asc,\rep,\rmin,\max,\ealm) \mbox{ to } (\asc,\rep,\rmin,\max-1,\ealm-1),
	\end{align*}
	and satisfies
	\begin{align*}\zero(s)&=\zero(h_4(s))+\chi(\ealm(s)=0),\\ 
	\rpos(s)&=\rpos(h_4(s))-\chi(\Prm(s)_{\rpos(s)}=\max(s)+1).
	\end{align*}
\end{lemma}
\begin{proof}
	For any $s\in\D_4$, replace the entry $\ealm(s)$ on the $(\max(s)+1)$-th position by $\max(s)$. Define $h_4(s)$ as the resulting sequence.
\end{proof}
By taking the change of statistics into account, we further divide the subset $\D_5$ into three disjoint subsets, i.e., $D_{5}=\D_{5,1}\dot\cup \D_{5,2}\dot \cup\D_{5,3}$ where 
\begin{align*}
\D_{5,1}&=\{s\in\D_5:\min\{s_i,\max(s)+2\le i\le |s|\}\le\ealm(s)\},\\
&\qquad \dot \cup\{s\in\D_5:\min\{s_i,\max(s)+2\le i\le |s|\}=\ealm(s)+1,\rpos(s)\ge \ealm(s)+1\},\\
\D_{5,2}&=\{s\in\D_5:\min\{s_i,\max(s)+2\le i\le |s|\}=\ealm(s)+1,\rpos(s)= \ealm(s)\},\\
\D_{5,3}&=\{s\in\D_5:\min\{s_i,\max(s)+2\le i\le |s|\}\ge \ealm(s)+2\}.
\end{align*}
Note that by definition $\D_{5,2}=\CMcal{M}_{5,2}$.
\begin{lemma}\label{L:h5}
	There is a bijection 
	\begin{align*}
	h_5:\D_5\cap \A_n \rightarrow (\D_3\dot \cup \D_4 \dot \cup \D_5) \cap \A_n
	\end{align*}
	that satisfies $\zero(s)=\zero(h_{5}(s))+\chi(\ealm(s)=0)$ and transforms
	\begin{align*}
	(\asc,&\rep,\max,\ealm) \mbox{ to } (\asc,\rep,\max,\ealm-1),\\
	&(\rmin,\rpos)  \mbox{ to } (\rmin,\rpos)\mbox{ if } s\in\D_{5,1},\\
	&(\rmin,\rpos)  \mbox{ to } (\rmin,\rpos-1)\mbox{ if } s\in\D_{5,2},\\
	&(\rmin,\rpos)  \mbox{ to } (\rmin-1,\rpos-1)\mbox{ if } s\in\D_{5,3}.
	\end{align*}
\end{lemma}
\begin{proof}
	For any ascent sequence $s\in\D_5$, define $h_5(s)$ to be the sequence after increasing the entry on the $(\max(s)+1)$th position by one. 
\end{proof}
\subsection{Proof of Theorem \ref{T:5tuple}}\label{S:pf7}

We prove this by induction on the numbers $|s|-\max(s)$ and $|s|-\rmin(s)$ for all ascent sequences $s\in \A_n$. For the trivial case $|s|=\max(s)=\rmin(s)$, we have $\Phi(s)=s$.

For $|s|=\max(s)+1$ with $\ealm(s)=i$ and $\max(s)=p$, then $s$ has the form $(0,1,\ldots,p-1,i)$. Choose $\Phi(s)$ to be the sequence after removing the last $i$ and inserting it right after the first $i$ of $s$, i.e., $\Phi(s)=(0,1,\ldots,i-1,i,i,\ldots,p-1)$ and (\ref{E:thm7sta}) clearly holds.

Suppose that the septuple $(\asc,\rep,\zero,\max,\ealm,\rmin,\rpos)$ on ascent sequences $s\in\A_n$ with $|s|-\max(s)=N-1$ is equidistributed to $(\asc,\rep,\zero,\rmin,\rpos,\max,\ealm)$ on ascent sequences $s\in\A_n$ with $|s|-\rmin(s)=N-1$ under the bijection $\Phi$, we next show it also holds when $N-1$ is replaced by $N$.

For any ascent sequence $s\in\A_n$ with $|s|-\max(s)=N$. If $s\in \D_2$, then according to Lemma \ref{L:maxcase2}, $h_2(s)=(\ealm(s),s^*)$ with $s^*\in \A^*\cap \A_{n-1}$ and $|s^*|-\max(s^*)=N-1$. By induction hypothesis and Lemma \ref{L:rmincase2}, define 
$$\Phi(s)=f_2^{-1}(\ealm(s),\Phi(s^*))\in\T_2\cap \A_n,$$ 
which is a bijection between the sets $\D_2\cap\A_n$ and $\T_2\cap\A_n$ such that $|s|-|\max(s)|=|\Phi(s)|-|\rmin(\Phi(s))|=N$. Furthermore, it follows from Lemma \ref{L:rmincase2} and \ref{L:maxcase2} that (\ref{E:thm7sta}) is true between the subsets $\D_2\cap \A_n$ and $\T_2\cap \A_n$.

If $s\in \D_3 \cap \A_n$ with $|s|-\max(s)=N$, then by Lemma \ref{L:p2} and \ref{L:maxcase5}, let $\tilde{s}=(\phi_2\circ h_3)(s)\in \A_{n-1}\cap \A^*$ with $|\tilde{s}|-\max(\tilde{s})=N-1$. As a result, by induction hypothesis, Lemma \ref{L:p1} and \ref{L:rmincase5}, define
\begin{align}\label{E:phi3}
\Phi(s)=(f_3^{-1}\circ\phi_1^{-1}\circ\Phi \circ \phi_2\circ h_3)(s)\in \T_3\cap \A_n,
\end{align} 
which is a bijection between the sets $\D_3\cap\A_n$ and $\T_3\cap\A_n$ such that $|s|-|\max(s)|=|\Phi(s)|-|\rmin(\Phi(s))|=N$. In addition, $\Phi$ also satisfies (\ref{E:thm7sta}) because of Lemma \ref{L:p1}, \ref{L:rmincase5}, \ref{L:p2} and \ref{L:maxcase5}.

If $s\in \D_4\cap \A_n$ with $|s|-\max(s)=N$, then according to Lemma  \ref{L:maxcase4}, $h_4(s)\in\A_n\cap \CMcal{P}_2^c$. Furthermore by induction hypothesis, there is a bijection $\Phi$ between $\A_{n-1}\cap \A^*$ and itself with $|s|-\max(s)=|\Phi(s)|-\rmin(\Phi(s))=N-1$. Together with Lemma \ref{L:p1} and \ref{L:p2}, we find that $\phi_1^{-1}\circ \Phi\circ \phi_2$ is the bijection between the set $\A_n\cap\CMcal{P}_2$ and 
$\A_n\cap\CMcal{P}_1$ with $|s|-\max(s)=|\Phi(s)|-\rmin(\Phi(s))=N-1$. In view of the induction hypothesis on the set $\A_n$, it follows that the complement $\A_n\cap \CMcal{P}_2^c$ is in bijection with $\A_n\cap \CMcal{P}_1^c$ via $\Phi$ such that $|s|-\max(s)=|\Phi(s)|-\rmin(\Phi(s))=N-1$. Define 
\begin{align}\label{E:phi4}
\Phi(s)=(f_4^{-1}\circ \Phi \circ h_4)(s)\in \T_4\cap \A_n,
\end{align}
which is a bijection between the sets $\D_4\cap \A_n$ and $\T_4\cap \A_n$ such that $|s|-|\max(s)|=|\Phi(s)|-|\rmin(\Phi(s))|=N$. The bijection $\Phi$ satisfies (\ref{E:thm7sta}) according to Lemma \ref{L:rmincase3} and \ref{L:maxcase4}.

If $s\in \D_5\cap \A_n$ with $|s|-\max(s)=N$, we will define $\Phi:\D_5\cap \A_n\rightarrow \T_5\cap \A_n$ recursively.

If $\max(s)-\ealm(s)=2$, then $h_5(s)\in\D_3 \dot\cup \D_4$. In view of  (\ref{E:phi3}) and (\ref{E:phi4}) for the case when $s\in\D_3\dot\cup\D_4$ with $|s|-\max(s)=N$, we have $(\Phi\circ h_5)(s)\in \T_3 \dot\cup \T_4$. As a result, we take
\begin{align}\label{E:phi5}
\Phi(s)=(f_5^{-1}\circ \Phi\circ h_5)(s)\in  \T_5 \cap \A_n,
\end{align}
which is a bijection for the case $\max(s)-\ealm(s)=2$. By repeatedly using (\ref{E:phi5}), we can recursively define the bijection $\Phi:\D_5\cap \A_n\rightarrow \T_5\cap \A_n$ for other ascent sequences $s\in\D_5\cap \A_n$ with $\max(s)-\ealm(s)>2$. In addition, by combining Proposition \ref{L:f5} and Lemma \ref{L:h5}, we can recursively verify that for $1\le i\le 3$, $s\in\D_{5,i}$ if and only if $(\Phi\circ h_5)(s)\in f_5(\M_{5,i})$, i.e., according to (\ref{E:phi5}), $(f_5^{-1}\circ \Phi\circ h_5)(s)=\Phi(s)\in\M_{5,i}$. This implies that $\Phi$ satisfies (\ref{E:thm7sta}). 

To sum it up, for $2\le i\le 5$, the bijection $\Phi:\D_i\cap \A_n\rightarrow \T_i\cap \A_n$ satisfying (\ref{E:thm7sta}) for the case $|s|-\max(s)=N$ is constructed, under the assumption that (\ref{E:thm7sta}) is true when $|s|-\max(s)=N-1$. Together with the starting case when $|s|-\max(s)=1$, it follows by induction that (\ref{E:thm7sta}) holds, which finishes the proof. 
\qed

\section{A refined generating function}\label{S:refgf}
This section deals with refined enumerations of ascent sequences with respect to the Euler--Stirling statistics $\asc$, $\rep$, $\max$ and $\rmin$, with the purpose to establish symmetric distributions directly from the generating function. 

Since the new decomposition of ascent sequences in Section \ref{S:thm3} is parallel
to the one from \cite{fjlyz}, it makes no real difference which one we choose to derive the refined generating functions, the decomposition from \cite{fjlyz} or the new one in
Section \ref{S:thm3} of this paper. 
For convenience, we use the decomposition from \cite{fjlyz}
because some explicit computations were already done there;
we only need to point out the differences when the statistic $\rmin$ is included.

We adopt the notations from \cite{fjlyz}. Let $\A$ be the set of all ascent sequences, i.e., $\A=\A^*\cup \{s:s=(0,1,\ldots,|s|-1)\}$ and define 
\begin{align*}
F(t;x,y,w,u,z,v)&:=\sum_{\substack{s\in\A\\ \vert s \vert>\max(s)}}t^{\vert s\vert}x^{\rep(s)}y^{\max(s)}w^{\ealm(s)}u^{\asc(s)}z^{\zero(s)}v^{\rmin(s)},\\
G(t;x,y,w,u,z,v)&:=\sum_{s\in\A}t^{\vert s\vert}x^{\rep(s)}y^{\max(s)}w^{\ealm(s)}u^{\asc(s)}z^{\zero(s)}v^{\rmin(s)},\\
&\,\,=vtyz(1-vtuy)^{-1}+F(t;x,y,w,u,z,v).
\end{align*}
Let furthermore $a_p(t;x,w,u,z,v):=[y^p]F(t;x,y,w,u,z,v)$.

Here is the partition of the set $\A^*$ into disjoint subsets from \cite{fjlyz}: the first two subsets $\D_1$ and $\D_2$ are defined in Section \ref{Sec:7tuple}.
 \begin{align*}
 \CMcal{S}_{1}&:=\CMcal{D}_{1}=\{s\in\A^*: \vert s\vert=\max(s)+1\},\\
  \CMcal{S}_{2}&:=\CMcal{D}_{2}=\{s\in\A^*-\CMcal{D}_{1}: s_{\max(s)+2}\le \ealm(s)\},\\
 \CMcal{S}_{3}&:=\{s\in\A^*-\CMcal{D}_{1}: s_{\max(s)+1}< s_{\max(s)+2},\max(s)\notin\{s_i: \max(s)+2\leq i\leq |s|\}\},\\
 \CMcal{S}_{4}&:=\{s\in\A^*-\CMcal{D}_{1}: s_{\max(s)+1}< s_{\max(s)+2},\max(s)\in\{s_i: \max(s)+2\leq i\leq |s|\}\}.
 \end{align*}
 By definition $\D_3\dot \cup\D_4\dot\cup\D_5=\CMcal{S}_{3}\dot\cup \CMcal{S}_{4}$. For each above subset, we will calculate the corresponding generating function in order to formulate a functional equation of $F(t;x,y,w,u,z,v)$ as below.
 \begin{proposition}
 	The generating function $F(t;x,y,w,u,z,v)$ satisfies
 	\begin{align}\label{E:idf2}
 &\left(1-\frac{ry-1}{y(1-w)}\right)F(t;x,y,w,u,z,v)\notag\\
 	\notag&=\frac{xyzvt^2(y^2tuwv(1-z)+z(y-yr+1))}{(1-ytu)(1-ytuvw)(y-yzr+z)}-\frac{tx}{1-w}F(t;x,wy,1,u,z,v)\\
 	\notag&\quad+(tux+y^{-1}-tu)\left(\frac{wy(1-z)+z(y-yr+1)}{(1-w)(y-yzr+z)}\right)
 	F(t;x,y,1,u,z,v)\\
 		&\quad+\frac{y^2u^2vt^2z(1-v)(tux+y^{-1}-tu)}{1-ytu}
 	\left(\frac{y^2tuvw(1-z)+z(y-yr+1)}{(1-ytuvw)(y-yzr+z)}\right)F(t;x,y,1,u,1,v),
 	\end{align}
 	where $r=t(u+x-xu)$.
 \end{proposition}
 \begin{proof}
 We omit the proofs of the generating function formulas for each subset $\CMcal{S}_i$ as they are direct extensions of the ones from \cite{fjlyz}.
 For the first two subsets $\CMcal{S}_1$ and $\CMcal{S}_2$, the generating functions
 are respectively:
 \begin{align}\label{eq:case1}
 &\sum_{s\in\D_1}t^{\vert s\vert}x^{\rep(s)}y^{\max(s)}w^{\ealm(s)}u^{\asc(s)}z^{\zero(s)}v^{\rmin(s)}
 =\frac{xyzvt^2(z+ytuwv-ytuzwv)}{(1-ytu)(1-ytuwv)},\\\intertext{and}
 \label{eq:case2}
 &\sum_{s\in\D_{2}}t^{\vert s\vert}x^{\rep(s)}y^{\max(s)}w^{\ealm(s)}u^{\asc(s)}z^{\zero(s)}v^{\rmin(s)}\notag\\
 &=\frac{tx}{1-w}(F(t;x,y,w,u,z,v)-F(t;x,yw,1,u,z,v))+tx(z-1)F(t;x,y,0,u,z,v).
 \end{align}
 For the second two subsets $\CS_3$ and $\CS_4$, the generating functions
 are respectively:
 \begin{align}
 &\sum_{s\in\CS_3}t^{\vert s\vert}x^{\rep(s)}y^{\max(s)}w^{\ealm(s)}u^{\asc(s)}z^{\zero(s)}v^{\rmin(s)}\label{eq:case3}\notag\\
 \notag&=
 (tux)\left(\frac{w+z-wz}{1-w}\right)F(t;x,y,1,u,z,v)-\frac{tux}{1-w}F(t;x,y,w,u,z,v)\\
 \notag&\quad -tux(z-1)F(t;x,y,0,u,z,v)\\
 &\quad +\frac{y^2u^3t^3vxz(1-v)(z(1-tuywv)+tuywv)}{(1-tuywv)(1-tuy)}F(t;x,y,1,u,1,v),\\
\intertext{and} \notag
&\sum_{s\in\CMcal{S}_4}t^{\vert s\vert}x^{\mathrm{rep}(s)}y^{\max(s)}w^{\ealm(s)}u^{\mathrm{asc}(s)}z^{\mathrm{zero}(s)}v^{\rmin(s)}\label{eq:case4}\\
 \notag&=\frac{(w+z-wz)(1-ytu)}{(1-w)y}F(t;x,y,1,u,z,v)\\
 \notag&\quad +\left(\frac{ytuvw(1-v)}{1-ytuvw}+z-zv\right)yvu^2t^2z F(t;x,y,1,u,1,v)\\
 &\quad-\frac{(1-tuy)}{(1-w)y}F(t;x,y,w,u,z,v)
 -\frac{(z-1)(1-tuy)}{y}F(t;x,y,0,u,z,v).
 \end{align}
The sum of all generating functions \eqref{eq:case1}--\eqref{eq:case4} equals $F(t;x,y,w,u,z,v)$, which leads to
\begin{align}\label{E:idf}
&\left(1-\frac{yt(x+u-ux)-1}{y(1-w)}\right)F(t;x,y,w,u,z,v)\notag\\
\notag&=\frac{xyzvt^2(z+ytuwv-ytuzwv)}{(1-ytu)(1-ytuwv)}-\frac{tx}{1-w}F(t;x,yw,1,u,z,v)\\
\notag&\quad+(z-1)(t(x+u-xu)-y^{-1})F(t;x,y,0,u,z,v)\\
\notag&\quad+\frac{w+z-wz}{1-w}(uxt+y^{-1}-ut)F(t;x,y,1,u,z,v),\\
&\quad+\frac{yu^2vt^2z(1-v)(z-ztuywv+tuywv)}{1-ytuvw}\left(1+\frac{yutx}{1-yut}\right)
F(t;x,y,1,u,1,v),
\end{align}
We next set $w=0$ and $r=t(x+u-xu)$ on both sides, yielding
\begin{align*}
F(t;x,y,0,u,z,v)&=\frac{y^2xz^2vt^2}{(1-ytu)(y-yzr+z)}
+\frac{z(ytux-ytu+1)}{y-yzr+z}F(t;x,y,1,u,z,v)\\
&\quad+\frac{y^2u^2vt^2z^2(1-v)(ytux-ytu+1)}{(1-ytu)(y-yzr+z)}F(t;x,y,1,u,1,v).
\end{align*}
Substituting the above expression for $F(t;x,y,0,u,z,v)$ in \eqref{E:idf}, we arrive at \eqref{E:idf2}.
 \end{proof}
By solving \eqref{E:idf2} for the case $z=1$, we deduce the generating function for the quadruple $(\asc,\rep,\max,\rmin)$ of statistics on ascent sequences, which is part of Theorem \ref{TC:int}.
 \begin{theorem}\label{T:int}
 	The generating function $\Gf(t;x,y,u,v)$ defined in \eqref{E:g2}
 	 is given by \eqref{E:g5}.
 \end{theorem}
\begin{proof}
 We apply the kernel method to \eqref{E:idf2}. Choose
 \begin{align*}
 1-\frac{yr-1}{y(1-w)}=0, \mbox{ that is},\, w=1+y^{-1}-r
 \end{align*}
 so that the left-hand-side of \eqref{E:idf2} becomes zero.
 Consequently the functional equation \eqref{E:idf2} is simplified to
 \begin{align}\label{E:F}
 \nonumber F(t;x,y,1,u,z,v)&=\frac{xzvt^2(1-yr)(ytuv(1-z)+z)}
 {(1-ytu)(1-tuv(y-yr+1))(tux+y^{-1}-tu)}\\
 \nonumber &\quad+\frac{tx(y-yzr+z)}{(y-yr+1)(tux+y^{-1}-tu)}F(t;x,y-yr+1,1,u,z,v)\\
 &\quad-\frac{yu^2vt^2z(1-v)(yr-1)(ytuv(1-z)+z)}{(1-ytu)(1-tuv(y-yr+1))}F(t;x,y,1,u,1,v).
 \end{align}
 We set $z=1$ on both sides, leading to
 \begin{align*}
 F(t;x,y,1,u,1,v)&=\frac{xvt^2(1-yr)}
 {(1-ytu)(1-tuv(y-yr+1))(tux+y^{-1}-tu)}\\
 &\quad+\frac{tx}{(tux+y^{-1}-tu)}F(t;x,y-yr+1,1,u,1,v)\\
 &\quad-\frac{yu^2vt^2(1-v)(yr-1)}{(1-ytu)(1-tuv(y-yr+1))}F(t;x,y,1,u,1,v),
 \end{align*}
 which can be simplified as
 \begin{align*}
 &F(t;x,y,1,u,1,v)\\&=\frac{xvt^2(1-yr)}
 {(1-ytu+tuv(yr-1))(1-ytuv)(tux+y^{-1}-tu)}\\
 &\quad+\frac{tx(1-ytu)(1-tuv(y-yr+1))}
 {(tux+y^{-1}-tu)(1-ytu+tuv(yr-1))(1-ytuv)}F(t;x,y-yr+1,1,u,1,v).
 \end{align*}
 Define $\delta_m:=r^{-1}-r^{-1}(1-yr)(1-r)^m$ so that $\delta_1=yw=y+1-yr$. By iterating the above equation, we conclude that 
 \begin{align}\label{E:f}
 \nonumber& F(t;x,y,1,u,1,v)\\
 \nonumber&=\sum_{m=0}^{\infty}\frac{rv(1-yr)(1-r)^m}
 {(x-xu+u(1-yr)(1-r)^m)(1-ytuv)}\\
 &\quad\times\prod_{i=0}^{m}\frac{x(1-(1-yr)(1-r)^i)(x-xu+u(1-yr)(1-r)^i)}
 {(x-u(x-1)(1-yr)(1-r)^i)(x-xu+u(1-rv)(1-yr)(1-r)^i)},
 \end{align}
 which is equivalent to \eqref{E:g5}.
 \end{proof}

\section{Transformations of basic hypergeometric series}\label{S:hyp}

For convenience, we recall some standard notions from the theory of basic
hypergeometric series, cf.\ \cite{gr}.

For indeterminates $a$ and $q$ (the latter is referred to as the base),
and non-negative integer $k$,
the basic shifted factorial (or $q$-shifted factorial) is defined as
\begin{equation*}
(a;q)_k:=\prod_{j=1}^k(1-aq^{j-1}).
\end{equation*}
This also makes sense for $k=\infty$, where the infinite product is viewed as
a formal power series in $q$ (whereas, viewed as an analytic expression in $q$,
we would need to insist on $|q|<1$, for convergence).
When dealing with products of $q$-shifted factorials, it is convenient to use
the following short notation,
\begin{equation*}
(a_1,\ldots,a_m;q)_k:=(a_1;q)_k\cdots(a_m;q)_k,
\end{equation*}
where again $k$ is a non-negative integer or $\infty$.

An ${}_\alpha\phi_\beta$ basic hypergeometric series with $\alpha$ upper
parameters $a_1,\ldots,a_\alpha$, and $\beta$ lower parameter $b_1,\ldots,b_\beta$,
base $q$ and argument $z$
is defined as
\begin{equation}\label{eq:bhyp}
{}_\alpha\phi_\beta\!\left[\begin{matrix}a_1,\ldots,a_\alpha\\
b_1,\ldots,b_\beta\end{matrix};q,z\right]:=
\sum_{k=0}^\infty\frac{(a_1,\ldots,a_\alpha;q)_k}
{(q,b_1,\ldots,b_\beta;q)_k}\left((-1)^kq^{\binom k2}\right)^{1+\beta-\alpha}z^k.
\end{equation}
The series in \eqref{eq:bhyp} (where the lower parameters are assumed to be chosen such
that no poles occur in the summands of the series)
terminates if one of the upper parameters, say $a_1$,
is of the form $q^{-n}$. Since $(q^{-n};q)_k=0$ for $k>n$,
the series in that case contains only finitely many non-vanishing terms.
If the series does not terminate, one usually imposes $|q|<1$.
See \cite[Sec.~1.2]{gr} for conditions under which the series converges.

One of the most important identities in the theory of basic hypergeometric series
is the Sears transformation~\cite[(III.15)]{gr},
\begin{equation}\label{tfsears}
{}_4\phi_3\!\left[\begin{matrix}q^{-n},a,b,c\\
d,e,abcq^{1-n}/de\end{matrix};q,q\right]=
\frac{(e/a,de/bc;q)_n}{(e,de/abc;q)_n}
\,{}_4\phi_3\!\left[\begin{matrix}q^{-n},a,d/b,d/c\\
d,aq^{1-n}/e,de/bc\end{matrix};q,q\right].
\end{equation}
In \eqref{tfsears}, $a,b,c,d,e$ and $q$ are indeterminates
and $n$ is a non-negative integer (which is responsible
that both ${}_4\phi_3$ series are actually finite sums and each
contains only $n+1$ non-vanishing terms).

While for non-terminating basic hypergeometric series in base $q$
we usually consider expansions around $q=0$, in this paper
(and more generally, when dealing with generating functions
of members of the Fishburn family)
we are dealing with power series in $r$, which
can be written as basic hypergeometric series in base $q=1-r$,
thus can be viewed as functions analytic around $q=1$.
We need to be cautious when we resort to non-terminating
identities for basic hypergeometric series.
The first part of the argument in the
proof of Theorem \ref{thm:4phi3tf}, as our main result in this section,
is similar to that used by Andrews and
Jel\'inek in \cite{aj} for establishing $q$-series identities around $q=1$.

\begin{proof}[Proof of Theorem \ref{thm:4phi3tf}]
Observe that both sides of
the identity converge as power series in $a$, thus
are analytic functions in $a$.
(This would not be true if $r$ would be replaced by $1-r$
with the series being expanded around $r=0$.)
Indeed, for each $m\ge 0$ the expansion of
$(1-a;1-r)_m$ in monomials $a^ir^l$ only involves terms with $i+l\ge m$
and each factor in the denominator of the series
has a non-vanishing constant term. Thus, in the expansion of the
series in the variables $a$ and $r$ the contribution of
coefficients for each monomial $a^ir^l$ is finite.

Now both sides of \eqref{tf43} agree for $a=(1-(1-r)^{-n})$ where
$n=0,1,2,\dots$ by the $(q,a,b,c,d,e)\mapsto
(1-r,(1-r)^j,b,c,d,e)$ special case
of the transformation in \eqref{tfsears}.
Since we have shown \eqref{tf43}
for infinitely many values of
$a$ accumulating at $a=-\infty$, i.e. $1-a=\infty$
(the transformation \eqref{tfsears}
itself is valid in the limiting case $n\to\infty$ (i.e. $q^{-n}\to\infty$)!),
by the identity theorem in complex analysis
the transformation \eqref{tf43}
is true for all $a$ in its domain of analyticity.
\end{proof}

\begin{remark}
It is interesting to notice that while the classical Sears transformation in
\eqref{tfsears} concerns a transformation between two terminating
${}_4\phi_3$ series in base $q$, valid as an identity around $q=0$,
the identity in Theorem~\ref{thm:4phi3tf}
concerns a transformation between two non-terminating ${}_4\phi_3$ series
in base $q=1-r$, valid as an identity around $r=0$ or, equivalently, $q=1$.
\end{remark}

We give two noteworthy specializations as immediate corollaries.
The first one is obtained by letting $a\to 1$ in \eqref{tf43}.
\begin{corollary}\label{cor:3phi2tf}
Let $b,c,d,e,r$ be complex variables, $j$ be a non-negative integer.
Then, assuming that none of the denominators in \eqref{tf32}
have vanishing constant term in $r$, we have the following transformation
of convergent power series in $r$:
\begin{align}\label{tf32}
&{}_3\phi_2\!\left[\begin{matrix}(1-r)^j,b,c\\
d,e\end{matrix};1-r,1-r\right]\notag\\
&=\frac{((1-r)/e;1-r)_j}{((1-r)bc/de;1-r)_j}\;
{}_3\phi_2\!\left[\begin{matrix}(1-r)^j,d/b,d/c\\
d,de/bc\end{matrix};1-r,1-r\right].
\end{align}
\end{corollary}
The second one is obtained by replacing $c$ by $d/c$ in \eqref{tf43}
and letting $d\to 0$.
\begin{corollary}\label{cor2:3phi2tf}
Let $a,b,c,e,r$ be complex variables, $j$ be a non-negative integer.
Then, assuming that none of the denominators in \eqref{2tf32}
have vanishing constant term in $r$, we have the following transformation
of convergent power series in $a$ and $r$:
\begin{align}\label{2tf32}
&{}_3\phi_2\!\left[\begin{matrix}(1-r)^j,1-a,b\\
e,(1-r)^{j+1}(1-a)b/ce\end{matrix};1-r,1-r\right]\notag\\
&=\frac{((1-r)/e,(1-r)(1-a)b/ce;1-r)_j}{((1-r)(1-a)/e,(1-r)b/ce;1-r)_j}
\;{}_3\phi_2\!\left[\begin{matrix}(1-r)^j,1-a,c\\
ce/b,(1-r)^{j+1}(1-a)/e\end{matrix};1-r,1-r\right].
\end{align}
\end{corollary}

The here obtained non-terminating basic hypergeometric transformations
of base $q=1-r$ (expanded around $r=0$)
are indeed powerful for proving equidistribution results for the Euler--Stirling statistics.
\begin{proof}[Proof of Theorem~\ref{TC:int}]
	Note that Theorem \ref{T:int} as part of Theorem \ref{TC:int} is already proved in Section \ref{S:refgf}. It remains to establish (\ref{E:sym1}) and (\ref{E:sym2}).
	
To show the symmetry $\Gf(t;x,y,u,v)=\Gf(t;x,v,u,y)$ is equivalent to showing
the identity
\begin{align}\label{tff0}
&\sum_{k=0}^\infty\frac{\left((1-yr)(1-r),\frac{u(1-yr)}{x(u-1)};1-r\right)_k(1-r)^k}
{\left(\frac{u(x-1)(1-yr)(1-r)}x,\frac{u(1-vr)(1-yr)(1-r)}{x(u-1)};1-r\right)_k}\notag\\
&=\frac{\left(1-\frac x{u(x-1)(1-yr)}\right)}{\left(1-\frac x{u(x-1)(1-vr)}\right)}
\sum_{k=0}^\infty\frac{\left((1-vr)(1-r),\frac{u(1-vr)}{x(u-1)};1-r\right)_k(1-r)^k}
{\left(\frac{u(x-1)(1-vr)(1-r)}x,\frac{u(1-vr)(1-yr)(1-r)}{x(u-1)};1-r\right)_k}.
\end{align}
Identity \eqref{tff0} is readily verified by virtue of the $j=1$ and
\begin{alignat*}3
b&=(1-yr)(1-r),&\qquad c&=\frac{u(1-yr)}{x(u-1)},\\
d&=\frac{u(1-vr)(1-yr)(1-r)}{x(u-1)},& \qquad e&=\frac{u(x-1)(1-yr)(1-r)}x
\end{alignat*}
special case of Corollary~\ref{cor:3phi2tf}.

On the other hand, to show the bi-symmetry $\G(t;x,y,u,z)=\G(t;u,z,x,y)$,
in view of Theorem~\ref{T:gen}, is equivalent to showing the identity
\begin{align}\label{tff}
&\sum_{k=0}^\infty\frac{((1-zr)(1-yr);1-r)_k\left(1-\frac{u(1-yr)}{x(u-1)}\right)}
{\left(\frac{u(x-1)(1-yr)(1-r)}x;1-r\right)_k
\left(1-\frac{u(1-yr)}{x(u-1)}(1-r)^k\right)}(1-r)^k\notag\\
&=\frac{\left(1-\frac{u(1-yr)}{x(u-1)}\right)\left(1-\frac x{u(x-1)(1-yr)})\right)}
{\left(1-\frac{x(1-zr)}{u(x-1)}\right)\left(1-\frac u{x(u-1)(1-zr)})\right)}\notag\\
&\quad\;\times
\sum_{k=0}^\infty\frac{((1-zr)(1-yr);1-r)_k\left(1-\frac{x(1-zr)}{u(x-1)}\right)}
{\left(\frac{z(u-1)(1-zr)(1-r)}u;1-r\right)_k
\left(1-\frac{x(1-zr)}{u(x-1)}(1-r)^k\right)}(1-r)^k.
\end{align}
Now, identity \eqref{tff} is readily verified by virtue of the $j=1$ and
\begin{alignat*}3
a&=r(z+y-zyr),&\qquad b&=\frac{u(1-yr)}{x(u-1)},\\
c&=\frac{x(1-zr)}{u(x-1)},&\qquad e&=\frac{u(1-yr)(1-r)}{x(u-1)}
\end{alignat*}
special case of Corollary~\ref{cor2:3phi2tf}.
\end{proof}

%

\section{Final remarks}\label{S:fre}

It is worthwhile to mention that an explicit formula for the refined generating function
of the five Euler--Stirling statistics $\asc,\rep,\zero,\max,\rmin$ on ascent sequences can be derived from \eqref{E:F} and \eqref{E:f}. We have the following result:
\begin{theorem}\label{T:5gen} Let $r=t(x+u-xu)$. The refined generating function for the quintuple $(\asc,\rep,\zero,\max,\rmin)$ of Euler--Stirling statistics on ascent sequences is 
	\begin{align}\label{eq:5gen}
	&\quad \quad G(t;x,y,u,z,v):=\sum_{n=1}^{\infty}t^n\sum_{s\in\A_n}
	x^{\rep(s)}y^{\max(s)}u^{\asc(s)}z^{\zero(s)} v^{\rmin(s)}=\frac{vytz}{1-vytu}\notag\\
	&+\sum_{k=0}^{\infty}
	\frac{yr^2vxz(tuv+z(r-tuv)-tuv(1-z)(1-yr)(1-r)^k)(1-yr)(1-r)^k}
	{(x-ux+u(1-yr)(1-r)^k)(r-tuv+tuv(1-yr)(1-r)^{k+1})(x-u(x-1)(1-yr)(1-r)^k)}\notag\\
	&\quad\quad\times\prod_{i=0}^{k-1}\frac{x-x(1-rz)(1-yr)(1-r)^i}
	{x-u(x-1)(1-yr)(1-r)^i}\notag\\
	&+\sum_{k=0}^{\infty}
	\frac{yr^2u^2vtz(1-v)(tuv+z(r-tuv)-tuv(1-z)(1-yr)(1-r)^k)(1-yr)(1-r)^k}
	{(x-xu+u(1-yr)(1-r)^k)
		(r-tuv+tuv(1-yr)(1-r)^{k+1})(r-tuv+tuv(1-yr)(1-r)^{k})}\notag\\
	&\quad \quad \quad\quad\times\sum_{m=k}^{\infty}\frac{rv(1-yr)(1-r)^{m}}
	{(x-xu+u(1-yr)(1-r)^{m})}\notag\\
	&\quad \quad \quad\quad\times\prod_{i=k}^{m}\frac{x(1-(1-yr)(1-r)^{i})
		(x-xu+u(1-yr)(1-r)^{i})}
	{(x-u(x-1)(1-yr)(1-r)^{i})(x-xu+u(1-rv)(1-yr)(1-r)^{i})}\notag\\
	&\quad \quad \quad\quad\times\prod_{j=0}^{k-1}\frac{x-x(1-rz)(1-yr)(1-r)^j}
	{x-u(x-1)(1-yr)(1-r)^j}.
	\end{align}
\end{theorem}
\begin{proof}
	Note that an equivalent form of (\ref{E:F}) is 
	\begin{align*}
	F(t;x,y,1,u,z,v)&=\frac{tx(y-yzr+z)F(t;x,y-yr+1,1,u,z,v)}{(tux+y^{-1}-tu)(y-yr+1)}\\
	&\quad-\frac{txz(ytuv(1-z)+z)}{tux+y^{-1}-tu}F(t;x,y-yr+1,1,u,1,v)\\
	&\quad+z(ytuv(1-z)+z)F(t;x,y,1,u,1,v).
	\end{align*}
	Since the last two items contain a common factor $z(ytuv(1-z)+z)$, let
	\begin{align*}
	H(t;x,y,1,u,z,v):&=F(t;x,y,1,u,z,v)
	-\frac{tx(y-yzr+z)F(t;x,y-yr+1,1,u,z,v)}{(tux+y^{-1}-tu)(y-yr+1)}.
	\end{align*}
	Then, the previous equation becomes
	\begin{align*}
	H(t;x,y,1,u,z,v)&=z(ytuv(1-z)+z)H(t;x,y,1,u,1,v).\\
	F(t;x,y,1,u,z,v)&=\frac{tx(y-yzr+z)F(t;x,y-yr+1,1,u,z,v)}{(tux+y^{-1}-tu)(y-yr+1)}\\
	&\quad+z(ytuv(1-z)+z)H(t;x,y,1,u,1,v).
	\end{align*}
	Consequently \eqref{E:F} can be rewritten as
	\begin{align}\label{E:h}
	H(t;x,y,1,u,1,v)
	\nonumber&=\frac{xvt^2(1-yr)}
	{(1-ytu)(1-tuv(y-yr+1))(tux+y^{-1}-tu)}\\
	&\quad+\frac{yu^2vt^2(1-v)(1-yr)}{(1-ytu)(1-tuv(y-yr+1))}F(t;x,y,1,u,1,v).
	\end{align}
	By iterating the above equation, we find that, with $\delta_m=r^{-1}-r^{-1}(1-yr)(1-r)^m$,
	\begin{align*}
	&F(t;x,y,1,u,z,v)\\&=\sum_{k=0}^{\infty}
	z(\delta_k tuv(1-z)+z)H(t;x,\delta_k,1,u,1,v)
	\prod_{i=0}^{k-1}\frac{tx(\delta_i-\delta_izr+z)}
	{(tux+\delta_i^{-1}-tu)(\delta_i-\delta_i r+1)}.
	\end{align*}
	Substituting $H(t;x,\delta_k,1,u,1,v)$ by the right-hand-side of \eqref{E:h} and then plugging \eqref{E:f} into the equation  (after setting $y=\delta_k$), we obtain the formula
	for the generating function in \eqref{eq:5gen}.
\end{proof}
\begin{remark}\label{rem:eq}
Neither of the generating function formulas in \eqref{E:genG} or in \eqref{E:g5} is a  direct specialization of the formula \eqref{eq:5gen}, although equivalent forms of the two 	former formulas can be obtained by setting $v=1$, respectively $z=1$, in the latter one.
\end{remark}

The formula \eqref{eq:5gen} for the generating function
$G(t;x,y,u,z,v)$ is of theoretical interest; it is explicit but unfortunately rather complicated.
It seems very difficult to apply this formula in order to prove equidistribution results
by pure algebraic means (i.e., manipulations of series), although we know that $G(t;x,y,u,z,v)=G(t;x,v,u,z,y)$ holds, as a consequence of Theorem \ref{T:5tuple}.

\begin{openproblem}
Find a simpler form of the generating function $G(t;x,y,u,z,v)$ so that  $\CMcal{G}(t;x,y,u,z)$ and $\Gf(t;x,y,u,v)$ (given in Theorems \ref{T:gen} and \ref{TC:int}) are straightforward specializations of $G(t;x,y,u,z,v)$ at $v=1$ and $z=1$, respectively. Furthermore, prove the symmetry $G(t;x,y,u,z,v)=G(t;x,v,u,z,y)$ by transformations of basic hypergeometric series.
\end{openproblem}

We finally pose a conjecture on a symmetric equidistribution of
Euler--Stirling statistics on inversion sequences, which is analogous to
Theorem \ref{T:5tuple} but with the two statistics $\ealm,\rpos$
being removed, and $\A_n$ (the set of ascent sequences) being replaced
by $\CMcal{I}_n$ (the set of inversion sequences).

\begin{conjecture}\label{OP:2}
	There is a bijection $\Omega: \CMcal{I}_n\rightarrow \CMcal{I}_n$
	such that for all $s\in \CMcal{I}_n$, 
	\begin{equation*}
	(\asc,\rep,\zero,\max,\rmin)s=(\asc,\rep,\zero,\rmin,\max)\Omega(s).
	\end{equation*}
   Consequently for all $\pi\in \mathfrak{S}_n$,
   \begin{equation*}
   (\des,\iasc,\lmax,\lmin,\rmax)\pi=(\des,\iasc,\lmax,\rmax,\lmin)(b^{-1}\circ \Omega\circ b)(\pi),
   \end{equation*}
   where $b:\mathfrak{S}_n\rightarrow \CMcal{I}_n$ is a bijection due to Baril and Vajnovszki (see Theorem 1 of \cite{bv}).
\end{conjecture}

This has been verified by Maple up to $n=10$. Different from ascent sequences, a generating function formula for the quadruple $(\asc,\rep,\zero,\max)$ of Euler--Stirling statistics on inversion sequences remains unknown, but one for the pair $(\asc,\rep)$ of Eulerian statistics was established by Garsia and Gessel \cite{gg}:
In view of \eqref{dou:fo}, let
\begin{align*}
B_n(u,x)&:=\sum_{s\in\CMcal{I}_n}u^{\asc(s)}x^{\rep(s)}=\sum_{\pi\in\mathfrak{S}_n}u^{\des(\pi)}x^{\iasc(\pi)},\\
H_n(u,x)&:=\sum_{s\in\CMcal{I}_n}u^{\asc(s)}x^{n-1-\rep(s)}=\sum_{\pi\in\mathfrak{S}_n}u^{\des(\pi)}x^{n-1-\iasc(\pi)},
\end{align*}
then $$B_n(u,x)=x^{n-1}H_n(u,x^{-1}),$$ and there holds
\begin{align}\label{E:genascrep}
\sum_{n\ge 0}\frac{H_n(u,x)\,t^n}{(1-u)^{n+1}(1-x)^{n+1}}
=\sum_{k\ge 0}\sum_{m\ge 0}\frac{u^k x^m}{(1-t)^{km}},
\end{align}
which implies $H_n(u,x)=H_n(x,u)$, or equivalently, $B_n(u,x)=B_n(x,u)$ (see also \eqref{inv:sym}). One possible approach to solve Conjecture \ref{OP:2}
is to deduce an extension of \eqref{E:genascrep} by including the Stirling statistics
$\zero,\max,\rmin$ and to read the symmetry directly from the extended
generating function formula. 

While Theorem \ref{bij:sym} holds if $\A_n$ is replaced by $\CMcal{I}_n$ (see the following Proposition \ref{P:syminv1} which is a direct result of a bijection due to Baril and Vajnovszki \cite{bv}),
it currently seems that the proof of Theorem \ref{T:5tuple} cannot be modified
to affirm Conjecture \ref{OP:2}.
\begin{proposition}\label{P:syminv1}
	There is a bijection $\varrho:\CMcal{I}_n\rightarrow \CMcal{I}_n$ such that for any $s\in\CMcal{I}_n$,
	\begin{align*}
	(\asc,\rep,\zero,\max)s=(\rep,\asc,\rmin,\zero)\varrho(s).
	\end{align*}
\end{proposition}
\begin{proof}
	Baril and Vajnovszki (see Theorem 1 of \cite{bv}) constructed a bijection
	$b:\mathfrak{S}_n\rightarrow \CMcal{I}_n$ satisfying that for
	any $\tau\in\mathfrak{S}_n$,
	\begin{align*}
	(\des, \iasc,\lmin,\lmax,\rmax)\tau=(\asc,\rep,\max,\zero,\rmin)b(\tau).
	\end{align*}
	Let $\tau^c=(n+1-\tau_1)(n+1-\tau_2)\cdots (n+1-\tau_n)$ be the complement of $\tau$,  then for any $s\in\CMcal{I}_n$, let $\tau=b^{-1}(s)$ and we have
	\begin{align*}
	(\asc,\rep,\zero,\max)s&=(\des,\iasc,\lmax,\lmin)b^{-1}(s)\\
	&=(\des,\iasc,\lmax,\lmin)\tau\\
	&=(\iasc,\des,\rmax,\lmax)(\tau^{-1})^c\\
	&=(\rep,\asc,\rmin,\zero)\,b((\tau^{-1})^c),
	\end{align*}
	that is, by defining $\varrho(s)=b((\tau^{-1})^c)$ the proof is complete.
\end{proof}
If Conjecture \ref{OP:2} is true, then it follows from Proposition \ref{P:syminv1} that Conjecture \ref{conj1ref} also holds if $\A_n$ is replaced by $\CMcal{I}_n$. \\


\end{document}